\documentclass[11pt]{amsart}

\usepackage{amsmath}
\usepackage{amsthm}
\usepackage{graphicx} 
\usepackage{pxfonts}
\usepackage{txfonts}
\usepackage{url}
\usepackage{pgf}
\usepackage{color}
\usepackage{hyperref}
\usepackage{accents}
\usepackage{array}
\usepackage{multicol}

\theoremstyle{plain}
\newtheorem{introtheo}{Th\'eor\`eme}

\newtheorem{introprop}[introtheo]{Proposition}

\theoremstyle{plain}
\newtheorem{theo}{Th\'eor\`eme}[section]
\newtheorem{prop}[theo]{Proposition}
\newtheorem{lemme}[theo]{Lemme}

\theoremstyle{definition}
\newtheorem{defi}[theo]{D\'efinition}

\newtheorem{ques}[theo]{Question}

\def\term{\textbf}

\newcommand{\cat}{(\begin{smallmatrix}2&1\\1&1\end{smallmatrix})}

\newcommand{\D}{\mathbb{D}}

\newcommand{\Enl}{\mathrm{Enl}}

\newcommand{\F}{\mathcal{F}}
\newcommand{\fgeod}{(\phi_\Sigma^t)_{t\in\R}}
\newcommand{\fdts}{(\phi_{2,3,7}^t)_{t\in\R}}
\newcommand{\fttq}{(\phi_{3,3,4}^t)_{t\in\R}}

\newcommand{\Fs}{\mathcal{F}^s}
\newcommand{\Fu}{\mathcal{F}^u}

\renewcommand{\ge}{\geqslant}

\newcommand{\GLZ}{\mathrm{GL}_2(\Z)}

\renewcommand{\H}{\mathrm{H}}
\newcommand{\Hy}{\mathbb{H}^2}

\newcommand{\Hopf}{
	\begin{picture}(4.5,0)
	\put(0,-0.8){\includegraphics[height=4.2mm]{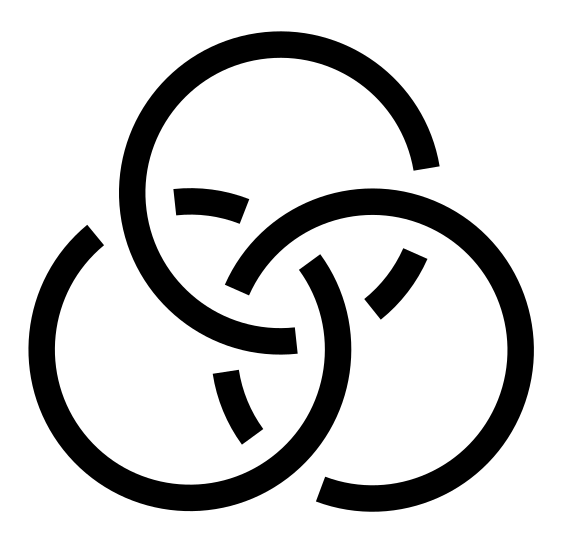}}
	\end{picture}}

\newcommand{\Huit}{
	\begin{picture}(3,0)
	\put(-1,-0.8){\includegraphics[height=4.2mm]{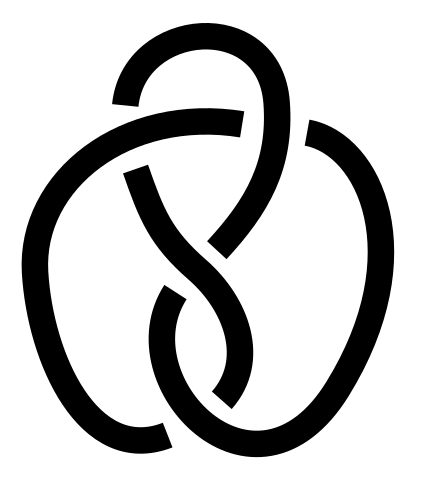}}
	\end{picture}}

\renewcommand{\le}{\leqslant}

\newcommand{\matL}{(\begin{smallmatrix}1&0\\1&1\end{smallmatrix})}
\newcommand{\matR}{(\begin{smallmatrix}1&1\\0&1\end{smallmatrix})}

\renewcommand{\P}{\mathbb{P}}

\newcommand{\Q}{\mathbb{Q}}

\newcommand{\R}{\mathbb{R}}

\newcommand{\Sph}{\mathbb{S}}

\newcommand{\SLZ}{\mathrm{SL}_2(\Z)}

\newcommand{\Spqr}{\Sigma_{p,q,r}}
\newcommand{\Sdts}{\Sigma_{2,3,7}}
\newcommand{\Sttq}{\Sigma_{3,3,4}}
\newcommand{\slk}{\mathrm{Enl}}

\newcommand{\T}{\mathbb{T}}

\newcommand{\U}{\mathrm{T}^1}

\newcommand{\vecgamma}{\accentset{\leftrightarrow}{\gamma}}

\newcommand{\Z}{\mathbb{Z}}

\begin{document}

\title[La courbe en huit et le n\oe ud de huit]{La courbe en huit sur les sphères à pointes et le n\oe ud de huit}
\date{16 juin 2022}

\author{Pierre Dehornoy}
\address{Université Grenoble Alpes, Institut Fourier, CNRS, Grenoble, France \text{(au moment de l'écriture)}}
\email[récente]{pierre.dehornoy@univ-amu.fr}
%%%%%%%%%%%%%%%%%%%%%%%%%%%%%%%%%%%%

\begin{abstract}
On montre que, dans le fibré unitaire tangent d'une orbisphère hyperbolique à points coniques d'ordre 3, 3, 4, le complémentaire du relevé de la plus courte géodésique périodique est homéomorphe au complémentaire du n\oe ud de huit dans la sphère de dimension 3. 
La preuve repose {\it in fine} sur des calculs d'enlacement. 
\end{abstract}

\maketitle

Quand on rencontre une variété réelle, compacte, sans bord, orientable de dimension 3, il n'est pas évident de l'identifier, même si des algorithmes généraux existent~\cite{Matveev}. 
Un cas très particulier est si la variété fibre sur le cercle avec fibre de genre~0 ou 1. 
Dans le premier cas on a affaire à~$\Sph^2\times\Sph^1$, et dans le second cas elle est déterminée par la classe d'isotopie de l'application monodromie, laquelle est représentée par une classe de conjugaison dans~$\GLZ$\footnote{
Si la fibre est de genre au moins 2, alors le problème est plus compliqué car la variété ne fibre en général pas de façon unique sur le cercle, et il faut résoudre le problème de savoir si deux suspensions distinctes donnent la même variété. 
Les {\it veering triangulations} semblent aujourd'hui un bon outil pour répondre théoriquement à cette question. 
}.

De la même fa\c con, identifier le complémentaire d'un n\oe ud dans une variété réelle de dimension~3 n'est en général pas facile (même pour le n\oe ud trivial dans la sphère réelle~$\Sph^3$ ! \cite{Haken, Kup}), sauf si on sait que ce complémentaire fibre sur le cercle avec fibre de genre 0 ou 1, auquel il suffit d'identifier la monodromie.  
Par exemple, notons \Huit~le n\oe ud de huit dans~$\Sph^3$. 
Un résultat folklorique stipule que \Huit~est fibré, avec fibre de genre 1, et monodromie donnée par (la classe de d'isotopie de l'action sur le tore de) la matrice~$\cat$. 
On en déduit que si un n\oe ud~$k$ dans une variété~$M$ est fibré, avec fibre de genre 1, et monodromie~$\cat$, alors $M\setminus k$ est homéomorphe \`a~$\Sph^3\setminus~\Huit$, et en particulier $M$ est obtenue à partir de~$\Sph^3$ par chirurgie sur \Huit. 

Dans ce texte, on donne deux exemples de cette situation. 
Pour $p, q, r$ trois entiers positifs satisfaisant~$\frac1p+\frac1q+\frac1r<1$, on note~$\Spqr$ l'unique orbisurface hyperbolique qui est une sph\`ere avec trois points coniques d'angles~$\frac{2\pi} p, \frac{2\pi} q, \frac{2\pi} r$. 
Le fibré unitaire tangent~$\U\Spqr$ est une variété réelle (sans point singulier) de dimension~3. 

\vspace{1mm}
\noindent
\begin{minipage}[t]{0.7\textwidth}
\hspace{3mm}Dans le cas~$p=2$, la plus courte géodésique périodique sur~$\Spqr$ est la ``hauteur'' issue du point d'ordre 2; on la note~$h$. 
Elle se relève dans~$\U\Spqr$ en un n\oe ud, not\'e~$\vec h$. 
Dans le cas~$2<p\le q\le r$, le plus courte géodésique périodique sur~$\Spqr$ est une courbe en forme de~$8$, qui tourne autour des deux points d'ordre~$p$ et $q$; on la note~$\gamma_8$. 
Une fois orientée (arbitrairement), elle se relève dans~$\U\Spqr$ en un n\oe ud, not\'e~$\vec\gamma_8$.

\begin{introtheo}\label{Th}
\begin{enumerate}
\item La variété~$\U\Sdts\setminus\vec h$ est homéomorphe \`a $\Sph^3\setminus~\Huit$.
\item La variété~$\U\Sttq\setminus\vec\gamma_8$ est homéomorphe \`a $\Sph^3\setminus~\Huit$.
\end{enumerate}
\end{introtheo}
\end{minipage}
\begin{minipage}[t]{0.3\textwidth}
\begin{picture}(0,0)(0,0)
\put(5,-13){\includegraphics[height=17mm]{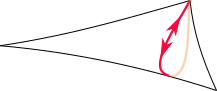}}
\put(8,-39){\includegraphics[height=30mm]{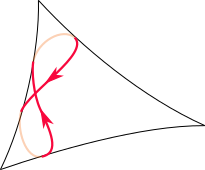}}
\put(32,-2){$h$}
\put(17,-28){$\gamma_8$}
\end{picture}
\end{minipage}

\vspace{3mm}
Ces résultats ne sont pas nouveaux au sens où une telle identification n'est pas surprenante (voir le paragraphe suivant\footnote{Ils avaient peut-être été identifiés par Thurston, même si ses notes ne donnent pas d'énoncé explicite~\cite[Chap 4]{ThurstonNotes}.}), et on peut probablement l'obtenir de plusieurs fa\c cons par des moyens élémentaires. 
En revanche la technique nous semble intéressante~: elle repose principalement sur des calculs de nombres d'enlacement. 
Le premier résultat avait été énoncé dans~\cite{GenusOne} avec une preuve proche\footnote{En réalité, le résultat avait d'abord été remarqué à l'aide de calculs informatiques avec la stratégie détaillée dans cette note, avant d'être publié avec une preuve plus directe, mais aussi plus ``parachutée''.}, mais le second n'a jamais été démontré ainsi. 
\`A la suite de l'écriture de cette note, une deuxième a été écrite~\cite{Matrix}, en anglais, avec une autre preuve du théorème~\ref{Th}, et trois cas supplémentaires. La technique est un peu différente: on y décrit explicitement la surface de genre 1 qui correspond à la fibre de la monodromie et on y calcule l'application de premier retour.

\medskip
\noindent{\bf Les chirurgies sur le n\oe ud de huit.} 
On peut se demander s'il y a beaucoup d'énoncés semblables au th\'eor\`eme~\ref{Th} à espérer. 
En effet, partant d'un n\oe ud~$k$ dans la sph\`ere tridimensionnelle r\'eelle~$\Sph^3$, on obtient une infinité de 3-variétés réelles en faisant des chirurgies sur ledit n\oe ud. 
Une telle chirurgie est définie par la classe d'isotopie de~$k$ et par un nombre rationnel~$b/a$; on note $\Sph^3(k, b/a)$ la variété ainsi obtenue. 
Thurston a démontré que, lorsque~$k$ est non trivial, toutes les variétés ainsi obtenues ---sauf un nombre fini--- sont hyperboliques. 
En particulier ce ne peut être un fibré unitaire tangent, ni même un fibré de Seifert. 
Dans le cas du n\oe ud de huit, Thurston a donn\'e la liste : si le coefficient de chirurgie est différent de~$0, \pm 1, \pm 2, \pm 3,\pm 4$, alors la variété est hyperbolique.

Naturellement, on se demande quelles sont les variétés obtenues dans les neuf cas spécifiques ci-dessus. 
Si on accepte le théorème de géométrisation, on sait que les variétés sont soit des suspensions d'automorphisme du tore bidimensionnel~$\T^2$ (correspondant à la géométrie Sol), soit des fibrés de Seifert (correspondant aux six autres géométries non hyperboliques en dimension~3).

Pour le n\oe ud de huit, le cas $b/a=0$ est le plus connu : le caractère fibré et le calcul de la monodromie impliquent que $\Sph^3(~\Huit, 0)$ est la variété obtenue par suspension de l'automorphisme~$\cat$ de~$\T^2$. 
Un calcul élémentaire montre que les autres cas donnent des sphères d'homologie rationnelle, et donc ne peuvent donner des variétés-suspension. 
Il s'agit donc de fibrés de Seifert. 
Le n\oe ud de huit étant chiral, les variétés $\Sph^3(~\Huit, b/a)$ et~$\Sph^3(~\Huit, -b/a)$ ne diffèrent que par l'orientation. 
Il ne reste alors que quatre cas à comprendre, et donc quatre fibrés de Seifert à identifier. 
Le théorème~\ref{Th} peut en fait être précisé en disant que les variétés $\Sph^3(~\Huit, \pm 1)$ sont hom\'eomorphes \`a~$\U\Sdts$ (\`a orientation pr\`es), tandis que les variétés $\Sph^3(~\Huit, \pm 3)$ sont hom\'eomorphes \`a~$\U\Sttq$ (nous ne donnerons pas la preuve de ce renforcement qui semble requérir une étude plus fine de la monodromie de $~\Huit$). 

Notons néanmoins que la stratégie de preuve du théorème~\ref{Th} va plutôt à l'envers~: elle permet d'identifier certains complémentaires de courbes particulières dans des fibrés unitaires tangents. 
Donc pour comprendre $\Sph^3(~\Huit, \pm 2)$ ou~$\Sph^3(~\Huit, \pm 4)$ par la même stratégie, il faudrait d'abord deviner la réponse, ce que nous n'avons pas fait. 

\medskip

\noindent{\bf Stratégie de preuve : flots et sections de Birkhoff.} Pour d\'emontrer le th\'eor\`eme~\ref{Th}, on va utiliser un outil dynamique canoniquement associé aux fibrés unitaires tangents : le flot géodésique. 
Celui-ci est défini sur le fibré unitaire tangent de toute surface orbifoldique. 
Lorsque celle-ci est hyperbolique, le flot est Anosov, ce qui fait que sa dynamique est assez bien comprise (bien que chaotique). 
Ici on s'intéresse donc aux deux flots géodésiques sur les variétés~$\U\Sdts$ et~$\U\Sttq$, notés respectivement~$\fdts$ et~$\fttq$.
Les courbes $\vec h$ et~$\vec \gamma_8$ sont des orbites périodiques de ces deux flots.
On va déduire le théorème~\ref{Th} des énoncés suivants~: 

\begin{introprop}\label{P:h}
\begin{enumerate}
\item L'orbite~$\vec h$ borde une section de Birkhoff pour~$\fdts$, not\'ee~$S_h$. 
\item La surface~$S_h$ est un tore \`a une composante de bord. 
\item L'application de premier retour sur~$S_h$ le long de~$\fdts$ est un diff\'eomorphisme d'Anosov n'ayant aucun point fixe \`a l'int\'erieur de~$S_h$. 
\end{enumerate}
\end{introprop}

\begin{introprop}\label{P:3huit}
\begin{enumerate}
\item Le triple de l'orbite~$\vec\gamma_8$ borde une section de Birkhoff pour~$\fttq$, not\'ee~$S_8$. 
\item La surface~$S_8$ est un tore \`a une composante de bord. 
\item L'application de premier retour sur~$S_8$ le long de~$\fttq$ est un diff\'eomorphisme d'Anosov n'ayant aucun point fixe \`a l'int\'erieur de~$S_8$. 
\end{enumerate}
\end{introprop}

Comme on le rappellera, une section de Birkhoff est une surface à bord transverse au flot et qui coupe toutes les orbites. 
L'existence d'une telle section implique que son bord est un n\oe ud fibré, et donc que le complémentaire du bord une variété qui fibre sur le cercle. 
Dans le cas présent, les fibres sont de genre~1, ce qui peut se voir sur la fa\c con dont le flot s'enroule autour du bord.  
Le calcul de la monodromie se fait en suivant l'application de premier retour le long du flot, qui est ici un difféomorphisme d'Anosov du tore, et donc donné par une classe de conjugaison dans~$\SLZ$. 
Déterminer une telle classe de conjugaison n'est pas forcément facile, mais ici la tache est simplifiée par le fait qu'ici on calcule uniquement la trace (en comptant les points fixes de l'application de premier retour) et qu'il n'y a qu'une classe de conjugaison de trace~$3$ dans~$\SLZ$, à savoir celle de~$\cat$.

Le point-clé qu'on va expliquer est que pour des flots d'Anosov, l'existence d'une section de Birkhoff à bord fixé, le calcul de son genre, et le nombre de points fixes de l'application de premier retour sont trois informations qui toutes se lisent dans des calculs d'enlacement entre orbites du flot considéré. 
En effet, les variétés~$\U\Sdts$ et~$\U\Sttq$ sont des sphères d'homologie rationnelles, et on peut donc y définir des notions d'enlacement (rationnel) entre n\oe uds disjoints, notées~$\Enl_{\U\Sdts}$ et~$\Enl_{\U\Sttq}$. 
En utilisant la direction stable du flot d'Anosov, on peut pousser toute orbite~$k$ sur un n\oe ud proche~$k^{+s}$ et obtenir une notion d'auto-enlacement.  
Les propositions~\ref{P:h} et~\ref{P:3huit} sont alors point à point conséquences des \'enonc\'es suivants. 

\begin{introprop}\label{P:hEnl}
\begin{enumerate}
\item Pour toute orbite périodique~$k$ de~$\fdts$, on a $\Enl_{\Sdts}(\vec h, k)<0$.
\item On a $\Enl_{\Sdts}(\vec h, \vec h^{+s}) = -1$.
\item Pour toute orbite périodique~$k$ de~$\fdts$ diff\'erente de~$\vec h$, on a $\Enl_{\Sdts}(\vec h, k)\le -2$.
\end{enumerate}
\end{introprop}

\begin{introprop}\label{P:3huitEnl}
\begin{enumerate}
\item Pour toute orbite périodiquee~$k$ de~$\fttq$, on a $\Enl_{\Sttq}(\vec\gamma_8, k)<0$.
\item On a $\Enl_{\Sttq}^{s}(\vec\gamma_8, \vec\gamma_8^{+s}) = -1/3$.
\item Pour toute orbite périodique~$k$ de~$\fttq$ diff\'erente de~$\vec\gamma_8$, on a $\Enl_{\Sttq}(\vec\gamma_8, k)\le -2/3$.
\end{enumerate}
\end{introprop}

Enfin, pour démontrer ces deux dernières propositions, on utilisera des outils précédemment développés~\cite{DPinski, Left} qui permettent, à l'aide d'objets appelés patrons, le calcul effectif des enlacements entre orbites périodiques quelconques des flots de type~$(\phi_{p,q,r}^t)_{t\in\R}$.

\medskip
Cette note se veut accessible au sens où on rappelle la plupart des éléments basiques. 
Elle se veut aussi courte et sera elliptique sur certains calculs. 

\medskip
\noindent\textbf{Remerciements.} 
Merci à Pierre Will avait deviné la partie (2) du théorème~\ref{Th} et d'être venu m'en parler; cette note est l'occasion de présenter des méthodes qui me sont chères.

%%%%%%%%%%%%%%%%%%%%%%%%%%%%%%%%%

\section{Orbisurfaces, fibrés unitaires tangents, et flots géodésiques}

On donne dans cette partie les définitions nécessaires à la compréhension de l'énoncé du théorème~\ref{Th}.

\subsection{Orbisurfaces et fibrés unitaires tangents}\label{S:Orbisurfaces}

De fa\c con générale, une orbivariété est un espace localement modelé, non sur~$\R^n$ comme les variétés standards, mais sur des quotients de~$\R^n$ par des groupes linéaires finis. 
Ici on s'intéresse à des orbivariétés particulières qui d'une part sont de dimension 2 et d'autre part sont ``bonnes'', au sens qu'elles admettent un revêtement simplement connexe. 

\begin{defi}
Une \term{orbisurface hyperbolique} est un espace métrique qui s'identifie au quotient du plan~$\Hy$ par un groupe fuchsien~$\Gamma$. 
Lorsque $\Gamma$ préserve l'orientation, l'orbisurface est dite orientable. 
\end{defi}

Comme les groupes fuchsiens admettent des domaines fondamentaux polygonaux dans~$\Hy$, on peut aussi voir une orbisurface hyperbolique comme un polygone hyperbolique dont on a apparié les côtés. 
Notons que, par définition des groupes fuchsiens, le stabilisateur de chaque point de~$\Hy$ est un groupe fini, et donc chaque point d'une orbisurface hyperbolique orientable admet un voisinage localement isométrique au quotient d'un disque hyperbolique par un groupe de rotation fini. 
En notant $\Gamma_k$ le groupe des rotations d'un disque d'angle multiple de~$\frac{2\pi}k$, 
on appelle alors \term{point conique} un point 

\noindent\begin{minipage}[t]{0.5\textwidth}
admettant un voisinage localement isométrique \`a~$\Hy/\Gamma_k$ avec $k\ge 2$. 
Les cas qui nous intéressent sont ceux présentés dans l'introduction, où $\Gamma$ est un groupe triangulaire, c'est-à-dire le groupe~$\Gamma_{p,q,r}$ des isométries de~$\Hy$ préservant l'orientation et un pavage par des triangles d'angles~$\frac{\pi}p, \frac{\pi}q, \frac{\pi}r$. 
Dans ce cas, le quotient, not\'e~$\Spqr$, est homéomorphe \`a une sphère, et métriquement il a trois points coniques d'angle totaux~$\frac{2\pi}p, \frac{2\pi}q, \frac{2\pi}r$. 
\end{minipage}
\begin{minipage}[t]{0.5\textwidth}
	\begin{picture}(0,0)(0,0)
	\put(5,-32){\includegraphics[width=.42\textwidth]{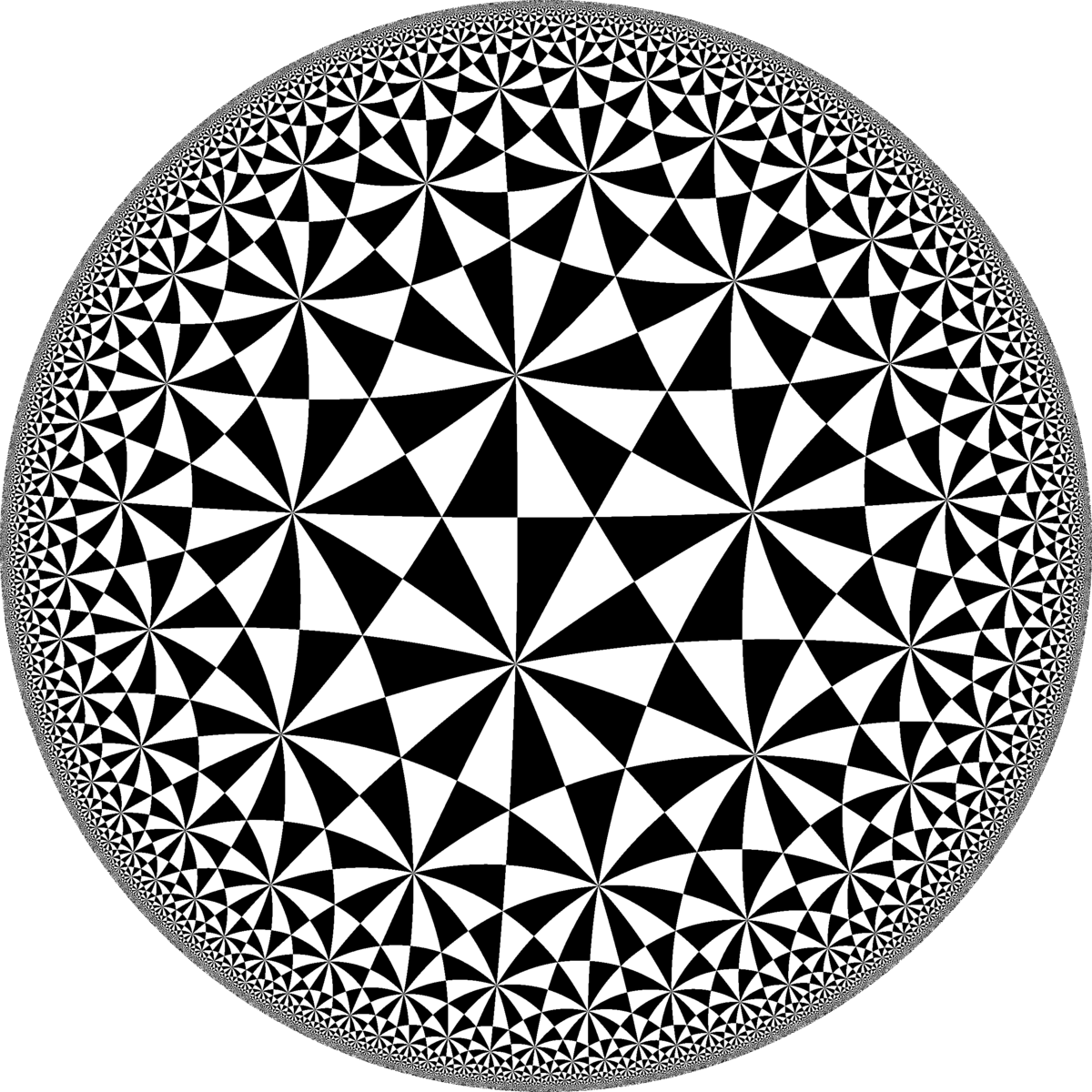}}
	\put(50,-32){\includegraphics[width=.42\textwidth]{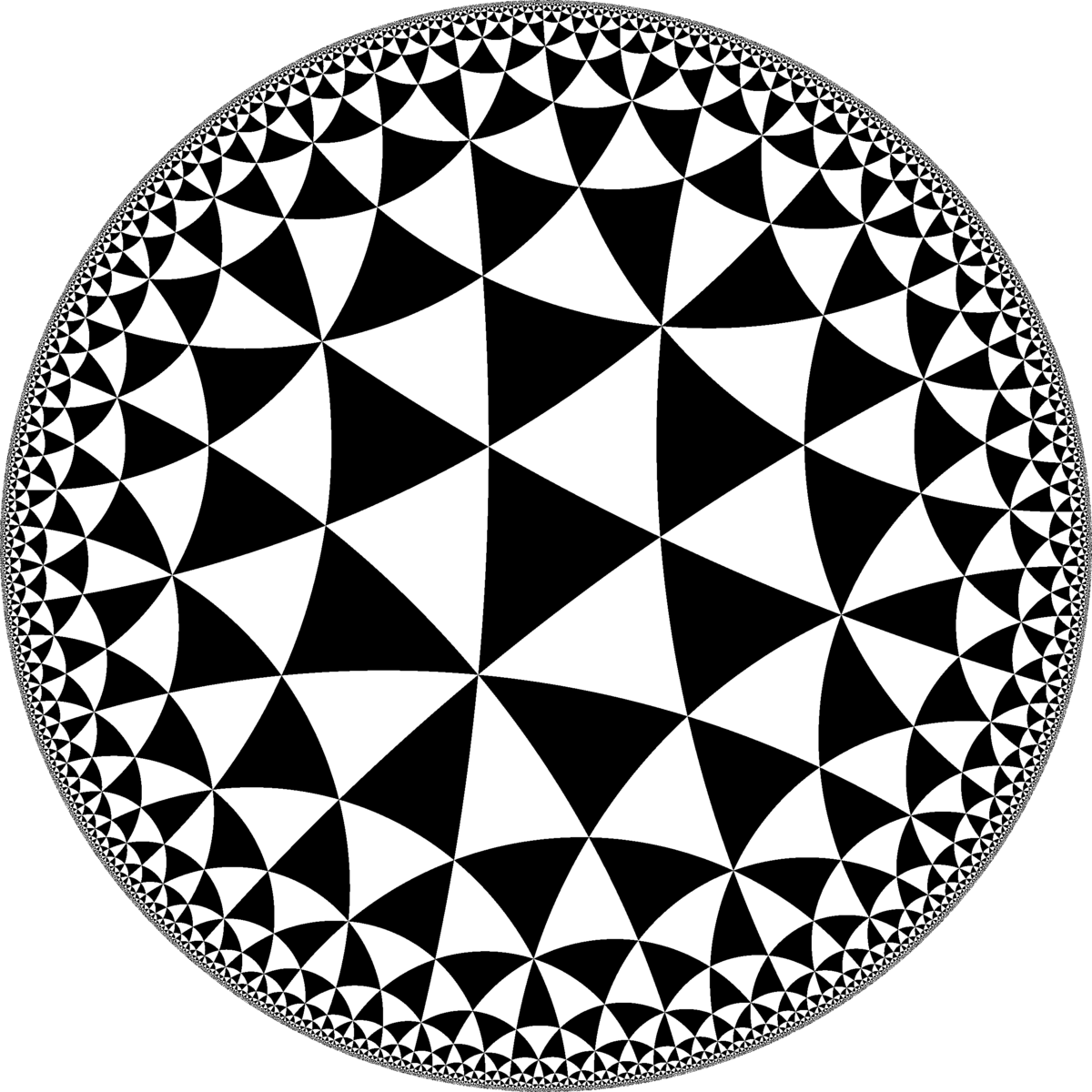}}
	\end{picture}
\end{minipage}

\medskip
Le fibré unitaire tangent d'une orbivariété peut être défini localement. 
On va plutôt utiliser une définition globale, plus simple, mais adaptée uniquement au cas des bonnes orbivariétés. 
Le fibré unitaire tangent~$\U\Hy$ est le fibré en cercle sur~$\Hy$ consitué des vecteurs tangents \`a~$\Hy$ et de norme~$1$. 
Topologiquement, il est homéomorphe au tore plein~$\Hy\times\Sph^1$. 
Si $\Gamma$ est un groupe fuchsien, il agit sur~$\Hy$ par isométries, et donc l'action s'étend \`a~$\U\Hy$. 

\begin{defi}
\'Etant donné une orbisurface hyperbolique $\Sigma=\Hy/\Gamma$, son \term{fibré unitaire tangent}, noté~$\U\Sigma$, est la variété tridimensionnelle~$\U\Hy/\Gamma$.
\end{defi}

Le fait que $\U\Sigma$ est une variété (c'est-à-dire avec des cartes modelées sur~$\R^3$ et non sur des quotients finis de~$\R^3$) même lorsque $\Sigma=\Hy/\Gamma$ n'en est pas une s'explique ainsi : une orbisurface hyperbolique orientable n'est pas une variété à cause des points coniques, lesquels proviennent des points de~$\Hy$ ayant un stabilisateur non trivial. 
Comme ce stabilisateur est un groupe fini de rotation, son action sur~$\U\Hy$ bouge les vecteurs tangents par rotation, donc elle

\noindent
\begin{minipage}[t]{0.7\textwidth}
est sans point fixe, ce qui signifie que~$\Gamma$ agit sur~$\U\Hy$ sans point fixe, et donc le quotient~$\U\Hy/\Gamma$ est bien une variété de dimension~3. 
Plus précisément, si~$U_p$ est un disque autour d'un point conique d'angle~$\frac{2\pi}p$, alors~$U_p$ s'identifie \`a $\D^2/\Gamma_p$.
Dans ce contexte $\Gamma_p$ agit sur le tore plein~$\U\D^2$ (représenté à gauche sur la figure ci-contre) par vissage comme suit~: $\bar k.(re^{i\theta},e^{i\phi})=(re^{i(\theta+\frac{2k\pi}p)}, e^{i(\phi+\frac{2k\pi}p)})$. 
Le quotient~$\U\D^2/\Gamma_p$ (représenté à droite) est aussi un tore plein, mais le feuilletage par fibre devient un feuilletage de Seifert du type $(p,1)$. 
\end{minipage}
\begin{minipage}[t]{0.3\textwidth}
	\begin{picture}(0,0)(0,0)
	\put(5,-32){\includegraphics[width=.84\textwidth]{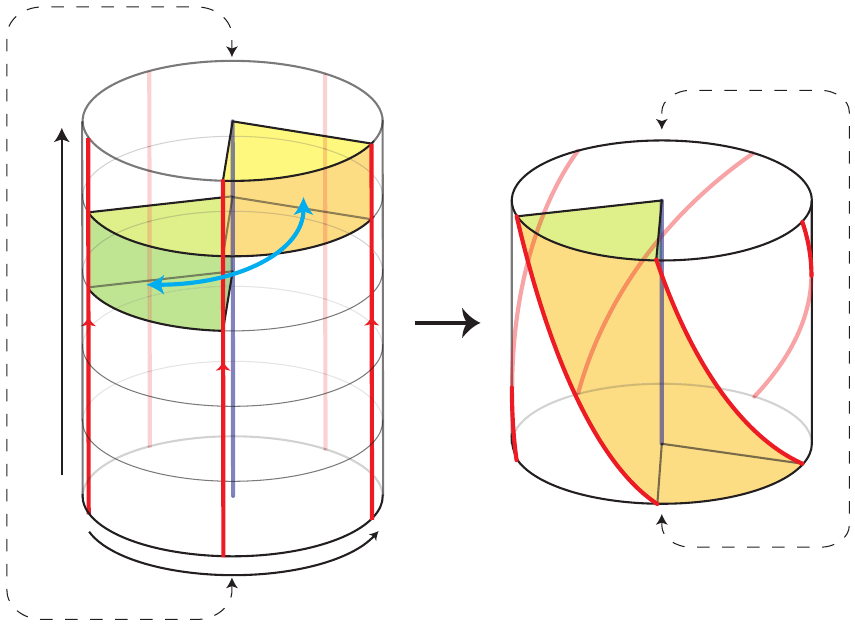}}
	\end{picture}
\end{minipage}

\subsection{Flots géodésiques}
\'Etant donn\'e métrique riemannienne sur une surface~$\Sigma$, le flot géodésique~$\fgeod$ sur~$\U\Sigma$ peut \^etre défini via l'équation~$\phi_\Sigma^t(\gamma(0), \dot\gamma(0))=(\gamma(t),\dot\gamma(t))$, pour toute géodésique~$\gamma$ parcourue à vitesse~$1$. 
Pour étendre la 

\noindent
\begin{minipage}[t]{0.7\textwidth}
définition aux orbisurfaces, il faut d'abord étendre la notion de géodésique. 
Dans notre cas, on peut simplifier en utilisant les géodésiques de~$\Hy$. 
En effet un groupe fuchsien agissant par isométrie, son action envoie les géodésiques de~$\Hy$ parcourues à vitesse 1 sur d'autres géodésiques parcourues à vitesse~$1$. 

\begin{defi}
Soit~$\Sigma=\Hy/\Gamma$ une orbisurface hyperbolique, le \term{flot géodésique sur~$\U\Sigma$} est le flot~$(\phi_{\Sigma}^t)_{t\in\R}$ sur~$\U\Sigma=\U\Hy/\Gamma$ obtenu par passage au quotient du flot géodésique sur~$\U\Hy$.
\end{defi}

\hspace{4mm}En particulier, les orbites périodiques de~$(\phi_{\Sigma}^t)_{t\in\R}$ sont les relevés des géodésiques périodiques (orientées) sur~$\Sigma$, lesquelles correspondent d'ailleurs aux classes de conjugaison dans~$\Gamma$.
\end{minipage}
\begin{minipage}[t]{0.3\textwidth}
	\begin{picture}(0,0)(0,0)
	\put(5,-45){\includegraphics[width=.84\textwidth]{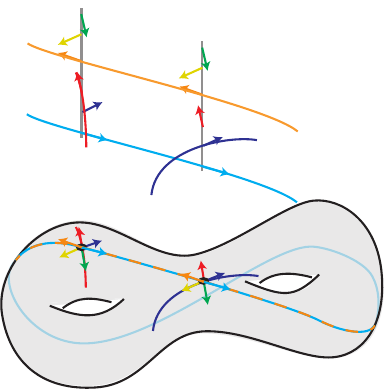}}
	\end{picture}
\end{minipage}

Par commodité, on note~$\fdts$ et~$\fttq$ les flots géodésiques sur~$\U\Sdts$ et~$\U\Sttq$ respectivement. 

\subsection{Flots d'Anosov}
La définition classique de flot d'Anosov tient en l'existence d'une décomposition invariante du fibré tangent \`a la variété en sommes de trois fibrés, l'un tangent au flot, le second exponentiellement contracté par la différentielle en temps positif, et le troisième exponentiellement contracté en temps négatif. 

On va donner une autre définition -- topologique -- dont on sait à présent~\cite{Shannon} qu'elle est équivalente à la définition classique pour les flots transitifs en dimension 3, mais qui a l'avantage d'être plus flexible vis-à-vis des opérations de chirurgie dont on va parler. 

\medskip
\noindent\begin{minipage}[t]{0.7\textwidth}
\begin{defi}
Un flot~$(\phi^t)_{t\in\R}$ sur une variété compacte sans bord~$M$ de dimension~3 est dit \term{topologiquement d'Anosov} s'il existe deux feuilletages bidimensionnels~$\F^s$, $\F^u$ de~$M$ qui sont invariants par~$(\phi^t)_{t\in\R}$, dont les feuilles sont transverses et se coupent le long des orbites de~$(\phi^t)_{t\in\R}$, tels que $\F^s$ est exponentiellement contracté par~$(\phi^t)_{t\in\R}$ pour $t\to+\infty$, et que $\F^u$ est exponentiellement contracté par~$(\phi^t)_{t\in\R}$ pour~$t\to-\infty$. 
Les feuilletages~$\F^s$ et $\F^u$ sont appelés les \term{feuilletages stable faible} et \term{instable faible} respectivement. 
\end{defi}
\end{minipage}
\begin{minipage}[t]{0.3\textwidth}
	\begin{picture}(0,0)(0,0)
	\put(2,-30){\includegraphics[width=.98\textwidth]{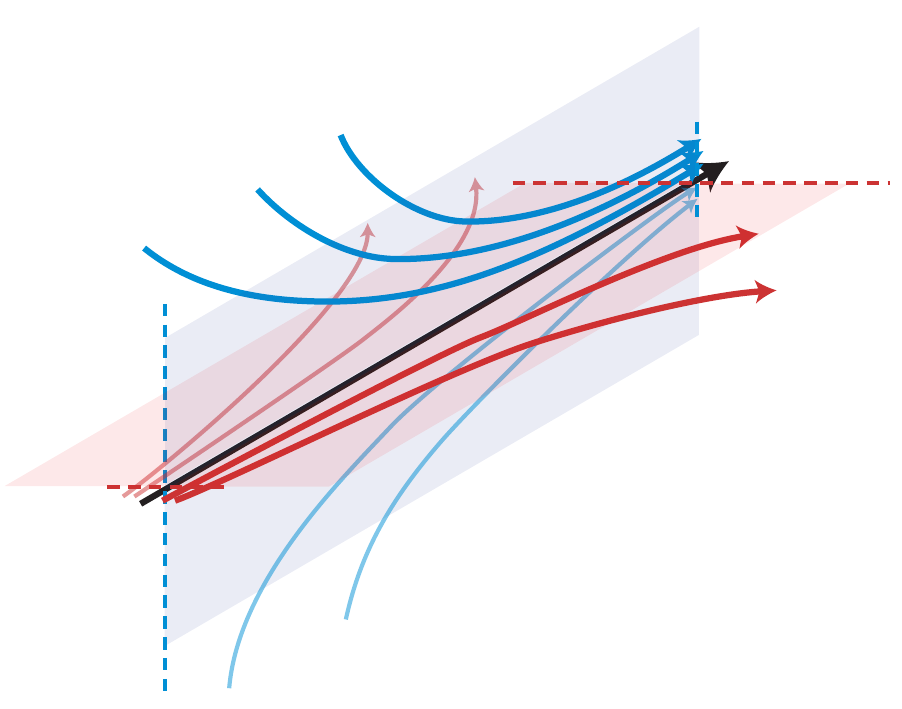}}
	\end{picture}
\end{minipage}

\medskip

On peut noter qu'être topologiquement d'Anosov est une propriété du feuilletage orienté de dimension~1 défini par les lignes de flot plus que du champ de vecteurs engendrant le flot.

\smallskip
\noindent\begin{minipage}[t]{0.7\textwidth}
\hspace{3mm}
Les flots géodésiques sur les orbisurfaces hyperboliques sont historiquement les premiers exemples de flots (topologiquement) d'Anosov~\cite{Anosov}. 
La feuille stable faible contenant les vecteurs tangents à une géodésique consiste en l'ensemble des vecteurs tangents à des géodésiques ayant la même extrémité positive dans~$\partial\Hy$, tandis que la feuille instable faible consiste en l'ensemble des vecteurs tangents à des géodésiques ayant la même extrémité négative. 
\end{minipage}
\begin{minipage}[t]{0.3\textwidth}
	\begin{picture}(0,0)(0,0)
	\put(10,-28){\includegraphics[width=.68\textwidth]{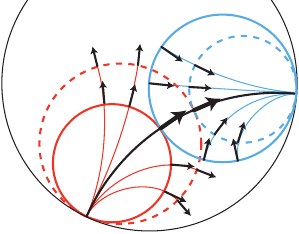}}
	\end{picture}
\end{minipage}

%%%%%%%%%%%%%%%%%%%%%%%%%%%%%%%%

\section{Chirurgies et sections de Birkhoff}

On donne dans cette partie les définitions nécessaires à la compréhension des propositions~\ref{P:h} et~\ref{P:3huit}, ainsi que la réduction du théorème~\ref{Th} à celles-ci. 

\subsection{\'Eclatement et chirurgie}\label{S:Eclatement}

On rappelle la convention standard~\cite{Rolfsen} : 
pour $M$ est une vari\'et\'e de dimension 3 et~$k\subset M$ un n\oe ud, on consid\`ere le \term{fibré unitaire normal} $\nu^1(k) := (\mathrm{T}M|_k/\mathrm{T}k)/\R^*_+$ : au-dessus de chaque point de $k$ il y a un cercle dont les points repr\'esentent les demi-plans bordés par la tangente \`a~$k$ en ce point.
Le fibr\'e~$\nu^1(k)$ est un tore bidimensionnel, avec des \term{méridiens} (orientés) canoniques $m$ donn\'es par les fibres (orientées) des points de~$k$. 
On appelle \term{longitude} de~$k$ une courbe~$\ell$ sur~$\nu^1(k)$ coupant chaque méridien une fois positivement, de sorte que $(m, \ell)$ forme une base  de~$\H_1(\nu^1(k))$. 
Celle-ci est directe quand on regarde de l'extérieur, ce qui correspond au fait qu'on veut que le m\'eridien, pouss\'e \`a l'ext\'erieur du n\oe ud, ait un enlacement~$+1$ avec le n\oe ud. 
Les classes d'homologie des autres longitudes de~$k$ sont alors les $[\ell+jm], j\in\Z$, c'est-\`a-dire les courbes de coordonn\'ees $(1, j)$. 

Remarquons que si $k$ est nul en homologie entière (en particulier si $M$ est une sph\`ere d'homologie entière), alors il y a une longitude canonique $\ell_{\mathrm{Seifert}}$ donn\'ee par le bord d'une surface de Seifert lisse de bord~$k$.

\'Etant donné une courbe simple sur~$\nu^1(k)$ de coordonn\'ees $(b,a)$, autrement dit homologue \`a $bm+a\ell$, sa \term{pente} est le rationnel~$b/a$. 
Ainsi, une pente $\infty$ correspond au m\'eridien~$m$ et une pente $0$ \`a la longitude~$\ell$ de départ. 
Les pentes entières correspondent alors aux courbes coupant une fois le m\'eridien, c'est-à-dire aux autres longitudes. 

Ensuite on consid\`ere la variété à bord $\overline{M\setminus k} := (M\setminus k) \cup \nu^1(k)$, appelée \term{\'eclatement normal de~$k$ dans~$M$}. 
Elle est hom\'eomorphe \`a $M\setminus U(k)$ o\`u $U(k)$ est un voisinage tubulaire de~$k$, qui est traditionnellement appelée un {forage de Dehn de~$M$ sur~$k$} . 
L'avantage de la pr\'esentation pr\'ec\'edente est qu'elle s'adapte mieux aux flots puisqu'on n'a retir\'e que $k$ et non tout un ouvert. 

Enfin, étant donné une pente~$b/a$, on choisit une famille de courbes de pente~$b/a$ qui feuill\`etent~$\nu^1(k)$. 
La \term{chirurgie de Dehn de pente $b/a$} consiste à coller un tore plein \`a $\overline{M\setminus k}$, de sorte que les méridiens du nouveau tore plein soient collés à la famille de courbes de pente $b/a$. 
On note $M(k, b/a)$ le r\'esultat. 

La définition ci-dessus s'étend sans peine si on remplace~$k$ par un entrelacs \`a un nombre arbitraire de composantes, pourvu qu'on précise une longitude et une pente pour chaque composante. 

\smallskip
On donne maintenant une présentation de la variété~$\U\Spqr$ par chirurgie.

On appelle pantalon une sphère à laquelle on a retiré trois disques; on note~$\P$ un tel pantalon. 
Pour $p,q,r$ des entiers au moins égaux à~2, les orbisurfaces~$\Sigma_{p,q,r}$ peuvent être topologiquement obtenues à partir d'un pantalon auquel on recolle, le long de ces trois composantes de bord, des disques quotientés~$\D^2/\Gamma_p, \D^2/\Gamma_q, \D^2/\Gamma_r$. 
La phrase précédente n'est pas rigoureuse, puisque topologiquement on a de toutes fa\c cons affaire à des sphères. 

Par contre, quand on passe au fibré unitaire tangent, le contenu n'est pas vide~: on peut obtenir le fibré unitaire tangent~$\U\Spqr$ à partir du fibré unitaire tangent d'un pantalon~$\U\P$, en recollant les fibrés unitaires tangents~$\U\D^2/\Gamma_p$, $\U\D^2/\Gamma_q$ et~$\U\D^2/\Gamma_r$. 
Comme on l'a expliqué en~\ref{S:Orbisurfaces}, ces derniers sont des tores, et pour déterminer un tel recollement il faut alors préciser sur quelle courbe on colle le méridien de ces tores. 
Notons que, le pantalon étant une surface à bord, son fibré unitaire tangent~$\U\P$ est topologiquement le fibré trivial~$\P\times\Sph^1$. 
Ce dernier peut être vu (de fa\c con à faire jouer un r\^ole symétrique aux trois bords) comme le complément d'un entrelacs de Hopf à trois composantes dans~$\Sph^3$, not\'e~\Hopf.
Par conséquent, le fibré unitaire tangent~$\U\Sigma_{p,q,r}$ peut être obtenu par chirurgies sur~\Hopf, avec des coefficients à déterminer en fonction de~$p,q,r$. 

\begin{theo}\cite{DPinski}\label{T:Chirurgie}
Pour $p, q, r$ des entiers supérieurs aux égaux à 2, le fibré unitaire tangent~$\U\Sigma_{p,q,r}$ est obtenu à partir de~$\Sph^3$ par chirurgies de pentes~$p-1, q-1, r-1$ sur les trois composantes d'un entrelacs de Hopf~\Hopf. 
\end{theo}

Notons que cette présentation permet aussi de calculer le groupe~$\H_1(\U\Sigma_{p,q,r}; \Q)$ suivant la procédure décrite par Saveliev~\cite[pp 3-4]{Saveliev}~: il s'agit du co-noyau de la matrice~$\begin{pmatrix} p&0&0&1\\0&q&0&1\\0&0&r&1\\1&1&1&-1\end{pmatrix}$. 
En particulier il s'agit d'un groupe abélien fini d'ordre~$pqr-pq-qr-pr$, donc $\U\Sigma_{p,q,r}$ est toujours une sphère d'homologie rationnelle. 
Ainsi on voit que $\U\Sdts$ est une sphère d'homologie entière, tandis que le groupe $\H_1(\U\Sttq; \Q)$ est d'ordre~$3$. 

\subsection{Flots et chirurgies}

Si $M$ est équipée d'un flot~$(\phi^t)_{t\in\R}$ de classe~$C^1$, il induit un feuilletage orient\'e de dimension~1. 
\'Etant donné une collection~$k$ d'orbites périodiques, le flot est évidemment bien d\'efini sur $M\setminus k$; 
via sa différentielle, il s'\'etend \`a $\nu^1(k)$, et donc \`a l'éclaté normal~$\overline{M\setminus k}$. 
Sur le tore~$\nu^1(k)$ qui est le bord de~$\overline{M\setminus k}$, le flot induit un feuilletage transverse aux cercles méridiens que sont les fibres des points de~$k$. 
Ainsi le flot étendu a un nombre de translation asymptotique qui correspond au rapport asymptotique entre le nombre de longitudes et le nombre de méridiens coupés par une orbite. 
Pour toutes les pentes sauf au plus une (qui correspond \`a ce nombre de translation asymptotique du feuilletage \'etendu \`a $\nu^1(k)$), le feuilletage sur $\nu^1(k)$ est transverse \`a une famille de m\'eridiens de pente~$b/a$, et donc un feuilletage de dimension 1 existe toujours apr\`es chirurgie. 
On peut alors trouver un flot $(\psi^t)_{t\in\R}$ engendrant ce feuilletage, et m\^eme on peut s'assurer que~$(\phi^t)_{t\in\R}$ et~$(\psi^t)_{t\in\R}$ ont les mêmes arcs d'orbites en dehors d'un voisinage tubulaire de~$k$. 

L'observation de Fried est la suivante~\cite{FriedAnosov}~: si un flot~$(\phi^t)_{t\in\R}$ est transitif et topologiquement d'Anosov et que $k$ est une orbite périodique (dont les feuilles stables et instables sont orientables), alors le flot étendu \`a~$\nu^1(k)$ a quatre\footnote{ou deux si les feuilles stables/instables ne sont pas orientables. Ce n'est pas un problème de traiter ce cas, mais on ne le rencontre pas ici avec les flots géodésiques.} orbites périodiques, dont deux correspondent au prolongement de la feuille stable de~$k$ et deux au prolongement de la feuille instable. 
La dynamique induite sur~$\nu^1(k)$ est alors de type hyperbolique,

\noindent\begin{minipage}[t]{0.6\textwidth}
\noindent
avec des orbites qui vont des deux orbites instables aux deux orbites stables.
Il y a alors une longitude~$\ell^s$ canonique donnée par la direction induite sur~$\nu^1(k)$ par les feuilles stables. 
Si $b/a$ est de la forme $1/n$, alors il y a une famille de courbes qui coupent chaque orbite stable (ou instable) exactement une fois. 
Quand on recontracte~$k$ le long de ces nouveaux méridiens, on récupère deux feuilletages ayant la bonne forme pour être topologiquement Anosov. 
La chirurgie de Dehn donne alors un flot $(\psi^t)_{t\in\R}$ qui est aussi topologiquement d'Anosov. 
\end{minipage}
\begin{minipage}[t]{0.4\textwidth}
	\begin{picture}(0,0)(0,0)
	\put(5,-33){\includegraphics[width=.9\textwidth]{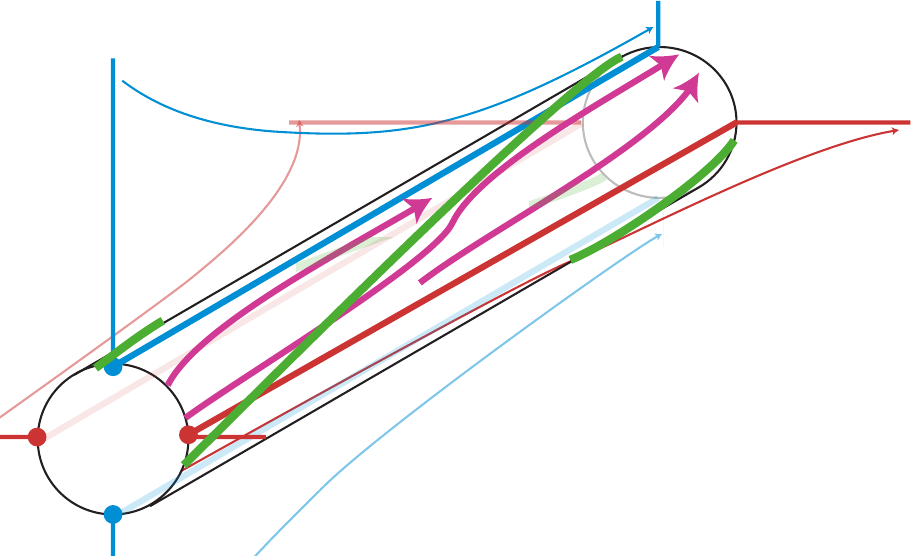}}
	\end{picture}
\end{minipage}

\vspace{1mm}
Que le flot obtenu soit effectivement d'Anosov, c'est-à-dire qu'il existe une structure différentiable pour lequel le flot est d'Anosov selon la définition classique est un fait absolument pas trivial, démontré récemment par Mario Shannon~\cite{Shannon}.

\subsection{Sections partielles et sections de Birkhoff}
%\noindent\begin{minipage}[t]{0.7\textwidth}
%Ici le terme~\emph{fibr\'e} s'entend en un sens un peu \'etendu par rapport \`a l'utilisation standard. 
%En effet la vari\'et\'e~$\U\Sigma_{0;3,3,4}$ est une sph\`ere d'homologie rationnelle, mais son homologie enti\`ere est isomorphe \`a $\Z/3\Z$. 
%Il y a donc des courbes qui ne bordent pas de chaine enti\`ere, mais dont le triple borde une chaine enti\`ere. 
%C'est ce qui se passe avec~$\vec\gamma_8$: comme on va le voir, c'est son triple qui borde un tore dans~$\U\Sigma_{0;3,3,4}$, ce dernier \'etant la fibre d'une fibration sur~$\Sph^1$. 
%Le tore en question a une composante de bord, son int\'erieur est plong\'e dans~$\U\Sigma_{0;3,3,4}$, mais son bord est immerg\'e : l'application induite $\bord\T^2\to\vec\gamma_8$ est de degr\'e~$3$. 
%\end{minipage}
%\begin{minipage}[t]{0.3\textwidth}
%	\begin{picture}(0,0)(0,0)
%	\put(9,-35){\includegraphics[width=.6\textwidth]{RevetementTriple1.pdf}}
%	\end{picture}
%\end{minipage}
%
\begin{defi}
\'Etant donn\'e une variété réelle~$M$, compacte, sans bord, orientable de dimension 3, et un champ de vecteurs non singulier~$X$ sur~$M$ de classe~$C^1$, engendrant un flot~$(\phi_X^t)_{t\in\R}$. 
Une \term{section de Birkhoff} pour~$(M,(\phi_X^t)_{t\in\R})$ est l'image d'une surface compacte \`a bord $i:S\to M$ telle que 

\begin{minipage}[t]{0.7\textwidth}
\begin{enumerate}
\item l'intérieur de $S$ est plongé dans $M$ et positivement transverse \`a~$X$, 
\item le bord de $S$ est immergé et tangent  \`a~$X$,
\item toute orbite du flot intersecte~$S$ en temps borné (autrement dit, il existe $T>0$ tel que $\phi_X^{[0,T]}(S)=M$).
\end{enumerate}
\end{minipage}
\begin{minipage}[t]{0.3\textwidth}
	\begin{picture}(0,0)(0,0)
	\put(3,-18){\includegraphics[width=.8\textwidth]{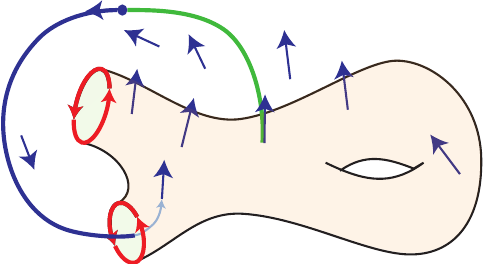}}
	\end{picture}
\end{minipage}
\end{defi}

La condition (2) implique que $i(\partial S)$ est une collection finie d'orbites périodiques de~$\phi_X$, et que $i|_{\partial S}$ est un rev\^etement fini. 
Lorsque $i|_{\partial S}$ est une bijection, on parle de  \term{section de Birkhoff plongée}. 
Si seules les conditions (1) et (2) sont remplies, on parle de \term{section partielle}. 

\'Etant donn\'e une section de Birkhoff~$i:S\to M$ pour~$(M, X)$, notons $k_1, \dots, k_n$ les composantes de l'entrelacs~$\partial S$. 
On note~$m_1, \dots, m_n$ les multiplicités des chaque composante de bord, qui correspondent aux degrés des revêtements $i|_{i^-1(k_1)}, \dots, i|_{i^-1(k_n)}$. 
Supposons donnée une longitude sur chacun des n\oe uds~$k_1, \dots, k_n$. 
On note alors~$b_i/a_i$ la pente induite par $S$ sur chaque composante~$k_i$. 
Dans la variété $M((k_1, b_1/a_1),\dots, (k_n, a_n/b_n))$, la surface~$S$ est alors une surface sans bord (puisque chaque composante de bord a été tuée par la chirurgie), et le flot~$(\phi_X^t)_{t\in\R}$ y est transverse \`a~$S$. 
En notant $f_S$ l'application de premier-retour sur~$S$ le long de~$(\phi_X^t)_{t\in\R}$, on a alors 

\begin{prop}\label{P:chirurgie}
Dans le contexte ci-dessus, $M((k_1, b_1/a_1),\dots, (k_n, a_n/b_n))$ est la variété-suspension de~$(S, f_S)$. 
\end{prop}

Après chirurgie, les orbites~$k_1, \dots, k_n$ modifiées coupent respectivement $m_1, \dots, m_n$ fois la surface~$S$.

On peut remarquer (ce qui ne sera pas utile pour la suite) que la donnée d'une section de Birkhoff $i:S\to M$ pour $(M,(\phi_X^t)_{t\in\R})$ induit une décomposition en livre ouvert rationnel de~$M$ dont la reliure est donnée par~$\partial S$ et une page par l'intérieur de~$S$. 
Pour obtenir les autres pages, on voudrait pousser~$S$ par le flot~$(\phi_X^t)_{t\in\R}$, mais le temps de premier retour n'est pas constant, ce qui est un problème. 
On pourrait mutliplier~$X$ par une fonction, mais au voisinage de la reliure, ce n'est toujours pas possible d'assurer un temps de premier retour constant. 
Ce qu'on peut alors faire, c'est modifier~$X$ dans un voisinage de la reliure en un champ~$X'$ dont les orbites proches de la reliure (plus précisément les orbites du champ étendu à l'éclatement normal de la reliure) sont des méridiens et restent transverses à~$S$. 
Alors on peut mutliplier~$X'$ par une fonction de sorte que le temps de premier du~$S$ soit constant, disons~$1$. 
Les autres pages du livre ouvert peuvent alors être obtenues en poussant~$S$ par le flot~$(\phi_{X'}^t)_{t\in\R}$. 

\subsection{Surfaces transverses et genre}\label{Genre}

Si $(\phi_X^t)_{t\in\R}$ est un flot d'Anosov sur une variété~$M$ et $S$ une section de Birkhoff (ou plus généralement une surface transverse, la condition de couper toutes les orbites n'est pas nécessaire ici), alors les feuilletages faibles~$\F^s$ et~$\F^u$ sont transverses à l'intérieur de~$S$, et y impriment donc des feuilletages~$\F^s\cap S$ et~$\F^u\cap S$ de dimension~1. 
Quitte à faire une petite isotopie, on peut aussi supposer que, si on éclate l'image du bord de~$S$ dans~$M$, la surface~$S$ est transverse au flot étendu, de sorte que les feuilletages~$\F^s\cap S$ et~$\F^u\cap S$ s'étendent à~$\partial S$, avec des singularités lorsque~$S$ coupe les directions stables ou instables des orbites de bord. 
On peut alors calculer la caractéristique d'Euler de~$S$ \`a partir de ces feuilletages étendus et du théorème de Poincaré-Hopf. 
Pour~$k$ une orbite dans le bord de~$S$, on définit l'\term{auto-enlacement de~$S$ par rapport \`a la direction stable en~$k$} comme le nombre (algébrique) d'intersection entre la courbe imprimée par~$\partial S$ sur~$\nu^1(k)$ et la direction stable~$\Fs$ sur~$\nu^1(k)$; on le note~$\slk_M^s(S, k)$.

\begin{prop}\label{P:genre}\cite{FriedFLP, DR}
Si $(\phi_X^t)_{t\in\R}$ est un flot d'Anosov sur une variété~$M$, $S$ une surface transverse à~$X$ dont le bord est inclus dans un entrelacs~$k_1\cup\dots\cup k_n$, alors on a 
\[\chi(S)=-\sum_{i=1}^n |\slk_M^s(S, k_i)|.\]
\end{prop}

\begin{proof}[Preuve]
Il suffit compter les singularités du feuilletage~$\F^s\cap S$ : elles sont toutes au bord de~$S$, proviennent d'une 

\noindent
\begin{minipage}[t]{0.7\textwidth}
intersection entre le bord de~$S$ et la direction stable de l'orbite de bord, et contribuent pour~$-1/2$ \`a la caractéristique d'Euler dans la formule de Poincaré-Hopf. 
Le fait qu'on n'ait pas ce facteur~$1/2$ dans la formule vient du fait que dans l'auto-enlacement, on ne compte qu'un seul côté de la variété stable en une orbite de bord, et non les deux. 
\end{minipage}
\begin{minipage}[t]{0.3\textwidth}
	\begin{picture}(0,0)(0,0)
	\put(2,-16){\includegraphics[width=.95\textwidth]{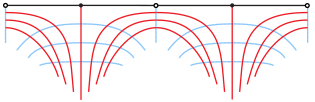}}
	\end{picture}
\end{minipage}
\vspace{-7mm}

\end{proof}

\subsection{Preuve du théorème~\ref{Th} \`a partir des propositions~\ref{P:h} et~\ref{P:3huit}}

Admettons la proposition~\ref{P:h}. 
%On va montrer la propriété suivante, légèrement plus forte que le théorème~\ref{Th}~(1)~:
%la variété~$M_1(\Huit)$ est homéomorphe~\`a
%

Puisque~$S_h$ est une section de Birkhoff pour~$\fdts$, l'orbite périodique~$\vec h$ forme un n\oe ud fibr\'e dans~$\U\Sdts$, dont la fibre est donnée par~$S_h$. 
Notons~$f_h:S_h\to S_h$ l'application de premier retour de long de~$\fdts$ correspondante. 
Puisque~$\U\Sdts$ est une sphère d'homologie entière, $\vec h$ porte une longitude de Seifert canonique, qui est d'ailleurs donnée par la direction du bord de~$S_h$ dans~$\nu^1(\vec h)$. 
Alors d'après la proposition~\ref{P:chirurgie} la variété~$\U\Sdts(\vec h, 0)$ est la variété-suspension de~$(S_h, f_h)$. 

Puisque~$S_h$ est un tore à bord transverse \`a~$\fdts$, le feuilletage~$\F^s\cap S_h$ est invariant par~$f_h$, et il est contracté uniformément. 
De même $\F^u\cap S_h$ est invariant et dilaté uniformément. 
On a déduit que~$f_h$ est une application d'un tore à une composante de bord qui préserve deux feuilletages transverses, contractant l'un et dilatant l'autre. 
En contractant la composante de bord en un point, $f_h$ induit alors une application pseudo-Anosov~$\bar f_h$ du tore~$\bar S_h$. 
Or cette application~$\bar f_h$ a au plus une singularité, qui correspond à la compactification de~$\vec h=\partial S_h$ en un point. 
Comme $\bar f_h$ est un tore, l'indice de cette singularité est~$0$. 
Ainsi chacun des deux feuilletages a exactement deux demi-feuilles qui arrive sur la singularité, et donc celle-ci n'en est en fait pas une, c'est-à-dire que~$\bar f_h$ est en fait un difféomorphisme d'Anosov du tore, qui est donc donné par une classe de conjugaison dans~$\SLZ$. 

Le nombre du points fixes d'une application Anosov du tore correspond \`a la trace $-2$ de la classe de conjugaison correspondante. 
Ici on a donc une classe de conjugaison dans~$\SLZ$ de trace~3. 
Or les classes de conjugaison des matrices non périodiques dans~$\SLZ$ sont énumérées par les mots positifs en~$L=\matL$ et $R=\matR$, à permutation cyclique près (voir~\cite[Prop. 4.3]{Lorenz} pour une preuve de ce fait est en fait une conséquence de l'algorithme de division euclidienne).
En particulier, il n'y a qu'une classe de conjugaison entière en trace~$3$, celle de~$RL=\cat$. 
Par conséquent~$\U\Sdts(\vec h, 0)$ est la variété-suspension de~$(\T^2, \cat)$.
Comme cette dernière est aussi~$\Sph^3(~\Huit, 0)$, les variétés~$\U\Sdts(\vec h, 0)$ et~$\Sph^3(~\Huit, 0)$ co\"incident. 
En particulier les complémentaires~$\U\Sdts\setminus\vec h$ et $\Sph^3\setminus$\Huit~co\"incident, ce qui démontre le théorème~\ref{Th}. 
Notons que cela implique aussi que $\U\Sdts$ est obtenu à partir de $\Sph^3$ par chirurgie sur le n\oe ud~$~\Huit$. 
%Pour déterminer le coefficient exact, on a besoin d'un peu plus que la proposition~\ref{P:h}, mais on expliquera dans la section~\ref{CalculChirurgie} comment ce coefficient déduire de la proposition~\ref{P:hEnl}. 

\medskip
Pour la seconde partie, on applique le même raisonnement à l'orbite périodique~$\vec\gamma_8$ de~$\fttq$. 

%%%%%%%%%%%%%%%%%%%%%%%%%%

\section{Flots Anosov, sphères d'homologie et existence de sections de Birkhoff}

On donne à présent les éléments permettant de déduire les propositions~\ref{P:h} et~\ref{P:3huit} des propositions~\ref{P:hEnl} et~\ref{P:3huitEnl}.

\subsection{Critère de Schwartzman-Fried pour les flots Anosov}

Partant d'une variété compacte sans bord~$M$ munie d'un flot~$(\phi_X^t)_{t\in\R}$, d'une collection finie~$k_1\cup\dots\cup k_n$ d'orbites périodiques, de multiplicités et de pentes sur ces orbites périodiques, on peut se demander s'il existe une section de Birkhoff~$i:S\to M$ pour~$(M, X)$ dont le bord est~$k_1\cup\dots\cup k_n$, avec les multiplicités et les pentes prescrites. 
\'Evidemment, il y a une obstruction topologique : il faut que cette donnée de bord corresponde au bord d'une surface. 
Cela donne une condition de codimension~$b_1(M)$ que nous ne détaillerons pas (les cas qui nous intéressent ici étant dans des sphères d'homologie pour lesquelles la condition est vide). 
Prenant cette condition en compte on peut reformuler le problème ce la fa\c con suivante~: \'etant donné une classe d'homologie~$\sigma\in\H_2(M, k_1\cup \dots\cup k_n; \Z)$, existe-t-il une section de Birkhoff pour~$(\phi_X^t)_{t\in\R}$ de classe~$\sigma$? 

Une condition nécessaire est que $\sigma$ coupe positivement les classes d'homologie de toutes les orbites périodiques de~$(\phi_X^t)_{t\in\R}$. 
En effet, si l'intersection entre les classes d'homologie était négative, la section ne pourrait couper l'orbite en des points avec intersection positive. 
Cette condition n'est en général pas suffisante~:
pour cela, il faut introduire la notion de \emph{cycle asymptotique} qui sont des classes d'homologie associées aux mesures invariantes du flot. 
Une définition précise peut être donnée en créant des cycles à partir de bouts d'orbites dont la longueur tend vers l'infini refermés arbitrairement~\cite{Schwartzman}, ou bien à parti de mesures invariantes et en passant par les courants~\cite{Sullivan}. 
Dans le cas des flots d'Anosov (ou pseudo-Anosov) transitifs, la densité des orbites périodiques permet néanmoins de se passer de ces subtilités, et la condition sur les orbites périodiques se trouve être suffisante~:

\begin{theo}\cite{FriedGeometry} 
Soit $(\phi_X^t)_{t\in\R}$ un flot d'Anosov sur une variété compacte~$M$ de dimension~3, $k_1, \dots, k_n$ des orbites périodiques de~$\phi_X$ et $\sigma\in\H_2(M, k_1\cup\dots\cup k_n; \Z)$ une classe d'homologie relative entière. 
La classe~$\sigma$ contient une section de Birkhoff pour~$(\phi_X^t)_{t\in\R}$ si et seulement si, pour toute orbite périodique~$k$ du flot étendu~$(\phi_X^t)_{t\in\R}$ à l'éclaté normal~$\overline{M\setminus(k_1\cup\dots\cup k_n)}$, on a 
\[\langle[k], \sigma\rangle>0.\]
\end{theo}

\subsection{Enlacement et dualité dans les sphères d'homologie}

\begin{defi}
Soit $M$ une vari\'et\'e compacte sans bord de dimension 3 qui est une sph\`ere d'homologie rationnelle et~$k_1, k_2$ deux entrelacs disjoints dans~$M$. 
Soit $S_1$ une 2-chaine rationnelle $S_1$ de bord~$k_1$. 
L'\term{enlacement de~$k_1$ et~$k_2$} est le nombre d'intersection alg\'ebrique entre les chaines $[S_1]\in \H_2(M, k_1; \Q)$ et~$[k_2]\in\H_1(M\setminus K_1; \Q)$, c'est-à-dire~$\Enl_M(k_1, k_2)=\langle S_1, k_2\rangle$.
\end{defi}

L'enlacement est en fait symétrique (ce qui se voit mieux si on prend une définition à l'aide d'intersections en dimension~4). 
Si $k_1$ est nul en homologie enti\`ere, alors on peut prendre pour $S_1$ une 2-chaine entière, de sorte que l'enlacement est entier si l'un des deux entrelacs est nul en homologie enti\`ere. 

On peut faire le lien entre la définition précédente et l'auto-enlacement d'une surface bordée par les orbites périodiques d'un flot d'Anosov introduit en~\ref{Genre} :
si $S$ est une surface dont le bord~$k_1\cup\dots\cup k_n$ est fait d'orbites périodiques d'un flot d'Anosov, on peut considérer l'entrelacs $k_1^{+s}\cup\dots\cup k_n^{+s}$ obtenu en dépla\c cant $k_1\cup \dots\cup k_n$ le long de la direction stable du flot. 
Alors on a $\Enl_M(k_1\cup \dots\cup k_n, k_1^{+s}\cup\dots\cup k_n^{+s})=\sum_{i=1}^n\Enl_M^s(S, k_i)$. 

Pour $M$ une sphère d'homologie rationnelle, $\H_2(M, k_1\cup\dots\cup k_n; \Z)$ est de dimension~$n$ et sans torsion, avec une base donnée par $([S_1], \dots, [S_n])$, où~$S_1, \dots, S_n$ sont des surfaces de Seifert pour~$k_1, \dots, k_n$. 
La dualité d'Alexander implique que l'enlacement porte toute la cohomologie des complémentaires d'entrelacs :

\begin{theo}
Soit $M$ une sphère d'homologie rationnelle en dimension~$3$ et $k_1\cup \dots\cup k_n$ un entrelacs dans~$M$. 
Alors $\H_1(M\setminus k_1\cup \dots\cup k_n; \Q)$ est de dimension~$n$, avec une base donnée par $(\Enl(k_1, \cdot), \dots, \Enl(k_n, \cdot))$.
\end{theo}

\subsection{Critère de Schwartzman-Fried en termes d'enlacement}

Les observations précédentes permettent de traduire le critère de Schwartzman-Fried pour les flots d'Anosov dans les sphères d'homologie~:

\begin{theo}\label{T:SFEnl}
Soit $(\phi_X^t)_{t\in\R}$ un flot d'Anosov sur une variété compacte~$M$ de dimension~3 qui est une sphère d'homologie rationnelle, $k_1, \dots, k_n$ des orbites périodiques de~$(\phi_X^t)_{t\in\R}$ et $m_1, \dots, m_n$ des entiers relatifs tels que $m_1[k_1]+\dots+m_n[k_n]=0$ dans~$\H_1(M; \Q)$. 
Alors il existe une section de Birkhoff pour~$(\phi_X^t)_{t\in\R}$ de bord $m_1k_1\cup\dots \cup m_nk_n$ si et seulement si, pour toute orbite périodique~$k$ du flot étendu~$(\phi_X^t)_{t\in\R}$ à $\overline{M\setminus(k_1\cup\dots\cup k_n)}$, on a 
\[\sum_{i=1}^nm_i\,\Enl_M(k_i, k)>0.\]
\end{theo}

\subsection{Preuves de~\ref{P:hEnl}$\implies$\ref{P:h} et \ref{P:3huitEnl}$\implies$\ref{P:3huit}}

Admettons la proposition~\ref{P:hEnl}. 
On rappelle de la section~\ref{S:Eclatement} que la variété~$\U\Sdts$ est une sphère d'homologie entière. 

On travaille dans la variété~$\overline{\U\Sdts\setminus\vec h}$ dont on rappelle qu'elle est l'éclaté normal de~$\U\Sdts$ le long de l'orbite~$\vec h$. 
Dans cette variété, on dispose du flot géodésique étendu~$\fdts$. 
Les orbites périodiques de ce dernier correspondent aux orbites périodiques du flot initial dans~$\U\Sdts$, sauf que l'orbite~$\vec h$ donne naissance \`a quatre orbites périodioques dans son éclaté~$\nu^1(\vec h)$, deux stables~$\Fs\cap\nu^1(\vec h)$ et deux instables~$\Fu\cap\nu^1(\vec h)$. 
Alors, d'après les points (1) et (2) de la proposition~\ref{P:hEnl}, le critère de Schwartzman-Fried~\ref{T:SFEnl} est satisfait, et donc $-\vec h$ borde (avec multiplicité~$-1$, comme indiqué par le signe) une section de Birkhoff, notons-la~$S_h$.

Comme la surface~$S_h$ est transverse au flot~$\fdts$ et que ce dernier est d'Anosov, la caractéristique d'Euler~$\chi(S_h)$ est donnée par la proposition~\ref{P:genre}. 
D'après le point (2) de la proposition~\ref{P:hEnl}, l'auto-enlacement de~$S_h$ par rapport \`a la direction stable de~$\fdts$ vaut~$-1$. 
Cela signifie que la surface~$S_h$ (de bord~$\vec h$) coupe $-1$ fois le n\oe ud~$\vec h^{+s}$, et donc~$-1$ fois la direction stable du flot.
Par conséquent on a $\chi(S_h)=-1$. 
Or $S_h$ a une composante de bord, donc $S_h$ est un tore. 

Enfin l'application de premier retour sur~$S_h$ le long de~$\fdts$ est un difféomorphisme d'Anosov. 
Les points fixes de ce difféomorphisme correspondent aux orbites du bord de~$S$ qui s'auto-enlacent 1 fois, et les orbites périodiques coupant l'intérieur de~$S_h$ exactement une fois. 
Or le point~(3) de la proposition~\ref{P:hEnl} dit qu'aucune orbite différente de~$\vec h$ ne coupe~$S_h$ moins de 2 fois. 
Donc l'application de premier retour sur~$S_h$ n'a qu'un point fixe, correspond \`a l'orbite~$\vec h$ elle-même. 
Le premier retour est donc bien donné par la matrice~$\cat$.

\medskip
Pour ce qui est de \ref{P:3huitEnl}$\implies$\ref{P:3huit}, l'argument est presque le même. 
Il y a néanmoins une subtilité supplémentaire, liée au fait que~$\U\Sttq$ n'est pas une sphère d'homologie entière, mais seulement rationnelle~: en effet on a~$\H_1(\U\Sttq; \Z)=\Z/3\Z$. 
Par conséquent, ou bien~$[\vec\gamma_8]$ est nulle dans~$\H_1(\U\Sttq; \Z)$, ou bien elle ne l'est pas mais $3\vec\gamma_8$ l'est. 
Dans le premier cas, l'enlacement entre $\vec\gamma_8$ est n'importe quel n\oe ud serait entier. 
Or l'énoncé de la proposition~\ref{P:3huitEnl} affirme que son auto-enlacement relativement à la direction stable vaut~$-1/3$. 
Par conséquent, on est dans le second cas~: $[\vec\gamma_8]$ est une classe non nulle dans~$\H_1(\U\Sttq; \Z)$, mais $3[\vec\gamma_8]$ est nulle. 

On peut alors appliquer le critère de Schwartzman-Fried pour déduire que $3\vec\gamma_8$ borde une section de Birkhoff pour~$\fttq$, notée~$S_8$. 
Par la proposition~\ref{P:genre} la caractéristique de~$S_8$ est alors donnée par~$3\Enl(\vec\gamma_8, \vec\gamma_8^{+s})$, ce qui donne~$-1$, et donc $S_8$ est un tore à une composante de bord. 
Enfin, le calcul de nombre de points fixes pour le premier retour est~$1$, par le m\^eme raisonnement, et donc l'application de premier retour sur~$S_8$ le long de~$\fttq$ est aussi~$\cat$.

%%%%%%%%%%%%%%%%%%%%%%%%%%%%%%%%

\section{Flots géodésiques, patrons et calculs d'enlacement}

On démontre maintenant les propositions~\ref{P:hEnl} et~\ref{P:3huitEnl}. 
Pour cela il faut comprendre la topologie des orbites périodiques des flots géodésiques dans les variétés~$\U\Sdts$ et~$\U\Sttq$. 
Nous le faisons en deux étapes~: d'abord ramener toutes les orbites périodiques des flots géodésiques sur deux \emph{patrons}, c'est-à-dire des surfaces branchées munies d'un flot dont le plongement sera explicite, et enfin calculer (ou plutôt majorer) les enlacements entre orbites périodiques de ces patrons.

\subsection{Patrons pour les flots géodésiques}

Les flots d'Anosov sont des flots markoviens, au sens où on peut décomposer la variété en boites de flots ``rectangulaires'', de sorte que les bords de ces bo\^ites s'agencent particulièrement bien. 
En contractant ces bo\^ites le long de la direction stable, on obtient alors des ``rubans'' qui vont se recoller aussi particulièrement bien. 
C'est ainsi que Birman et Williams ont montré~\cite{BW} comment on peut construire une surface branchée, appelée \term{patron}, plongée dans la variété, munie d'un flot (en fait un semi-flot car le passé n'est pas défini de façon unique), telle que toute collection finie d'orbites périodiques du flot de départ est isotope à une collection finie d'orbites périodiques du patron. 
Pour un flot explicite, construire un tel patron requiert une partition de Markov explicite, ce qui n'est pas toujours facile à obtenir. 
Aussi le nombre de rubans sera égale au nombre de bo\^ites de la partition. 
De telles partitions ont été construites par Bowen-Series, puis par Pit~\cite{BS,Pit}. 
Il se trouve que leurs constructions laissent de côté les orbisphères triangulaires. 

Dans notre article~\cite{DPinski}, nous avons affaibli les contraintes sur les rubans pour obtenir un patron avec seulement

\noindent\begin{minipage}[t]{0.6\textwidth}
deux rubans, qui porte toute la topologie des orbites périodiques du flot géodésique~$(\phi_{p,q,r}^t)_{t\in\R}$ sur~$\U\Sigma_{p,q,r}$; on le note~$\mathcal{T}_{p,q,r}$. 
On rappelle (théorème~\ref{T:Chirurgie}) que la variété~$\U\Spqr$ est obtenue à partir de~$\Sph^3$ par chirurgies d'indices~$p,q,r$ sur les trois composantes d'un entrelacs de Hopf~\Hopf. 
Le patron~$\mathcal{T}_{p,q,r}$ y est plongé comme indiqué ci-contre, en s'enroulant autour des deux composantes de~\Hopf~correspondants aux points d'ordre $p$ et $q$. 
Le patron~$\mathcal{T}_{p,q,r}$ est une surface branchée, munie d'un semi-flot (le long du segment de branchement, le passé n'est pas uniquement défini). 
Il a les deux propriétés suivantes : d'une part on peut coder toutes ses orbites périodiques \`a l'aide de mots bi-infinis périodiques sur l'alphabet~$\{a,b\}$ -- $a$ pour un tour sur l'oreille de gauche et $b$ pour un tour sur l'oreille de droite -- et les mots admissibles sont \end{minipage}
\begin{minipage}[t]{0.4\textwidth}
	\begin{picture}(0,0)(0,0)
	\put(4,-46){\includegraphics[width=.85\textwidth]{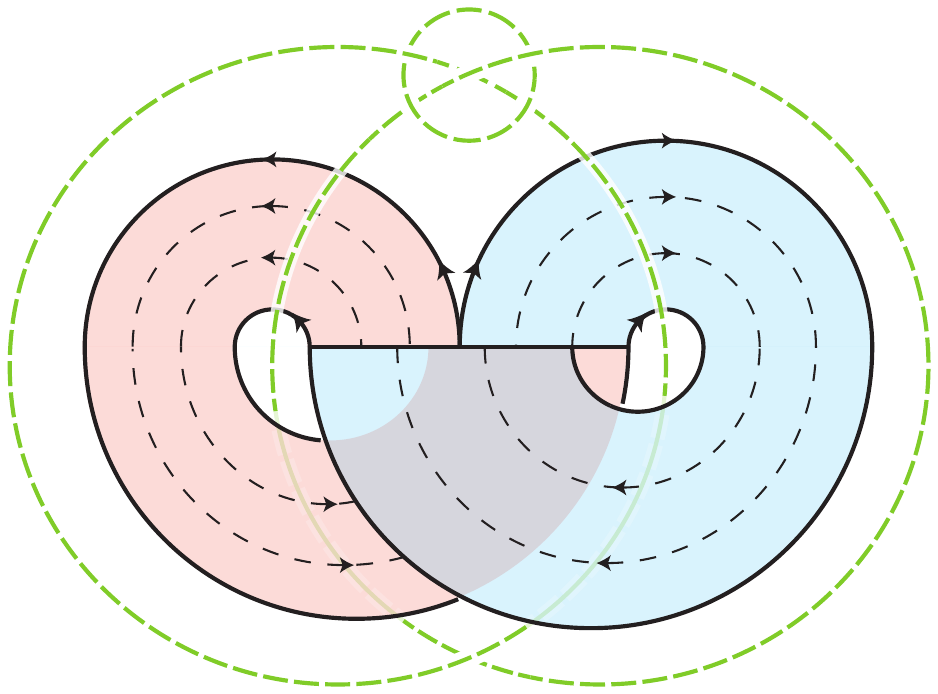}}
	\end{picture}
\end{minipage}

\vspace{1mm}
\noindent caractérisées par le fait qu'eux et leurs images par le décalage sont compris pour l'ordre lexicographique entre deux mots infinis explicites, appelés kneading sequences~\cite[table 1]{DPinski}. 
D'autre part ses orbites périodiques sont en bijection avec les orbites périodiques du flot géodésique~$(\phi_{p,q,r}^t)_{t\in\R}$. 

Le codage peut être compris sur un pavage de~$\Hy$ par des triangles hyperboliques d'angles~$2\pi/p, 2\pi/q, 2\pi/r$~: 
en dessinant un réseau~$\mathcal{R}_{p,q,r}$ de ``rond-points'' qui tournent trigonométriquement autour des points d'ordre~$p$ et horairement autour des points d'ordre~$q$, on voit que toute géodésique de~$\Hy$ peut être déformée sur~$\mathcal{R}_{p,q,r}$. 
Si on veut le faire de fa\c con canonique et équivariante, il faut \^etre soigneux~\cite[Section 3]{DPinski}. 
\`A la fin, le nombre de~$p$-ièmes de tour de la géodésique déformée autour d'un point d'ordre~$p$ donne une suite de~$a$, et le nombre de $q$-ièmes de tour autour d'un point d'ordre~$b$ donne une suite de~$b$. 

\noindent\begin{minipage}[t]{0.6\textwidth}
\hspace{3mm}
Dans le cas~$p=2, q=3, r=7$ qui nous intéresse, on voit que le réserau~$\mathcal{R}_{2,3,7}$ est constitué de ronds-points d'ordre~2 et 3. 
Les kneading sequences donnent des contraintes sur les mots admissibles. 
Plutôt que de les écrire, nous préférons expliquer leur signification géométrique. 
Par exemple il ne peut y avoir deux $a$ consécutifs, car cela correspondrait à une courbe qui s'enroule autour du point d'ordre~$2$ et non à une géodésique. 
De même il ne peut y avoir trois $b$ consécutifs. 
En réécrivant $x$ pour $ab$ et $y$ pour $abb$, on voit que toute géodésique est codée par une suite de~$x$ et de~$y$. 
La contrainte suivante interdit à une orbite de ``s'enrouler'' autour des points d'ordre~$7$, ce qui se traduit par le fait qu'on ne peut avoir plus de deux $x$ ou deux~$y$ consécutifs. 
Une fois ces premières contraintes prises en compte, on voit que les premières orbites périodiques du patron~$\mathcal{T}_{2,3,7}$ (ordonnées par nombre de tours autour du patron) sont codées par~$(xy)^\Z=(ababb)^\Z$, $(x^2y)^\Z=(abababb)^\Z$, $(xy^2)^\Z=(ababbabb)^\Z$, $(x^2y^2)^\Z=(abababbabb)^\Z$, $(x^2yxy)^\Z$, $(xyxy^2)^\Z$, $(x^2yxy^2)^\Z$, $(x^2y^2xy)^\Z$, $(x^2yx^2y^2)^\Z$, \end{minipage}
\begin{minipage}[t]{0.4\textwidth}
	\begin{picture}(0,0)(0,0)
	\put(4,-53){\includegraphics[width=.5\textwidth]{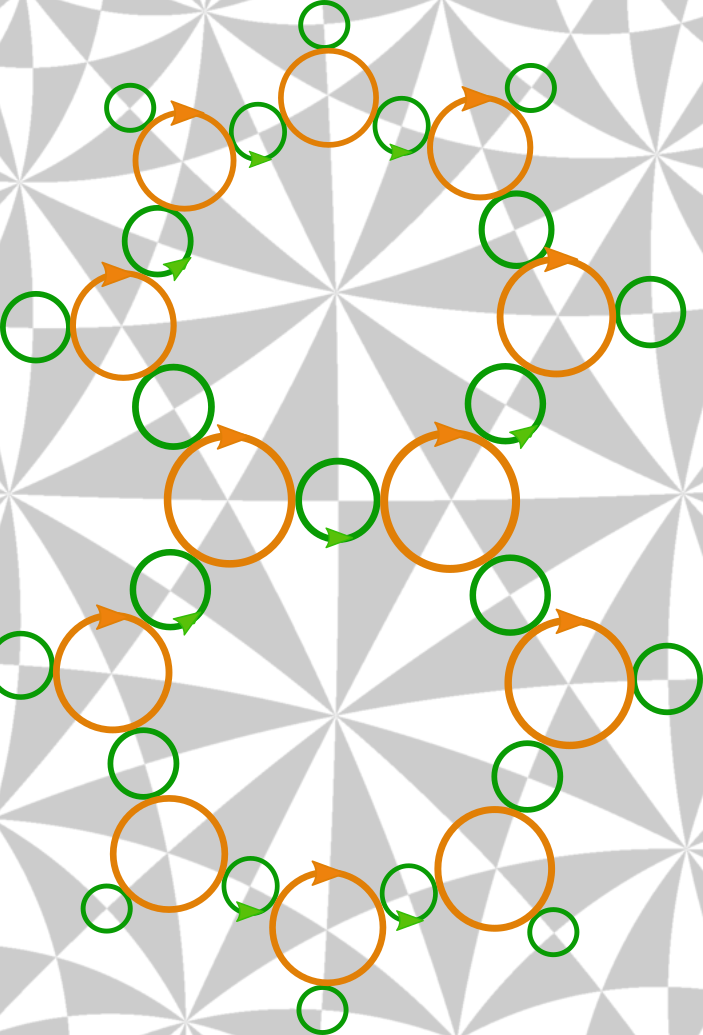}}
	\put(44,-66){\includegraphics[width=.3\textwidth]{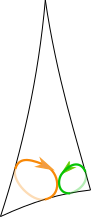}}
	\end{picture}
\end{minipage}
\vspace{-1mm}

\noindent
$(x^2y^2xy^2)^\Z$, etc. 
Il se trouve que les codes~$(xy)^\Z$, $(x^2y)^\Z$, $(xy^2)^\Z$ correspondent tous trois \`a la géodésique~$h$~\footnote{Les kneading sequences sont censées assurer l'unicité du codage. Dans ce cas elles interdisent les suites $(x^2y)^\Z$, $(xy^2)^\Z$, mais pour notre problème, ne pas avoir l'unicité n'est pas un problème, du moment que toute orbite périodique du flot géodésique est codée par au moins un mot.}.

%Pour~$\mathcal{T}_{2,3,7}$, les kneading sequences sont~$u_L=(0101011)^\infty$ et $v_R=11(01101)^\infty~\footnote{Ces mots peuvent sembler compliqué. Il nous semble que c'est le prix à payer pour avoir un codage avec uniquement deux lettres.}$, ce qui signifie qu'un mot infini~$w$ code une orbite de~$\mathcal{T}_{2,3,7}$ si et seulement si toutes ses images~$(\sigma^n(w))_{n\in\N}$ par le décalage satisfont~$u_L<\sigma^n(w)<v_R$ pour l'ordre lexicographique. 
%Par exemple, un mot ne peut avoir deux $0$ consécutifs car il serait alors inférieur à $u_L$, et il ne peut avoir trois $1$ consécutifs car il serait alors supérieur à $u_V$. 
%Les premiers mots périodiques admissibles par période croissante sont alors $(01011)^\infty$ qui correspond \`a l'orbite~$\vec h$, $(01011011)^\infty$
 
De fa\c con semblable, dans le cas~$p=3, q=3, r=4$, le réserau~$\mathcal{R}_{3,3,4}$ est constitué de deux ronds-points autour des points d'ordre~3. 
Une géodésique est encore codée par une suite de~$a$ et de $b$, cette fois avec jamais trois~$a$ ni trois~$b$ consécutifs. 
Les contraintes liées au point d'ordre~$4$ impliquent qu'un $ab^2$ ne peut être suivi d'un autre $ab^2$. 
De même un~$a^2b$ ne peut être suivi d'un autre $a^2b$. 
Avec cela, on voit que les premières orbites périodiques de~$\mathcal{T}_{3,3,4}$ sont codées par $(ab)^\Z$ et $(a^2b^2)^\Z$ qui correspondent justement aux deux relevés de~$\gamma_8$, selon l'orientation choisie. 
Les orbites suivantes correspondent aux codes~$(a^2bab)^\Z$, $(abab^2)^\Z$, $(a^2bab^2)^\Z$, $(a^2b^2ab)^\Z$, $(a^2ba^2b^2)^\Z$, $(a^2b^2ab^2)^\Z$, etc. 

\subsection{Formule d'enlacement dans~$\U\Spqr$}\label{Hopf}

Puisqu'on va chercher à calculer des nombres d'enlacement dans~$\U\Sdts$ et~$\U\Sttq$ et que ces deux variétés sont obtenues par chirurgie sur l'entrelacs~\Hopf~de~$\Sph^3$, on cherche une formule qui, \'etant donné deux n\oe uds dans~$\Sph^3$ disjoints de l'entrelacs de Hopf, relie l'enlacement de ces courbes avant et après chirurgie. 
Une telle formule n'est pas difficile à obtenir, il suffit de comprendre comment obtenir une 2-chaine bordée par un n\oe ud après chirurgie à partir d'une 2-chaine avant chirurgie. 
On note~$Q_{p,q,r}$ la forme bilinéaire symétrique sur~$\R^3$ donnée par la matrice $\frac1{pqr-pq-qr-pr}\left(\begin{smallmatrix}
qr-q-r & r & q\\
r & pr-p-r & p\\
q & p & pq-p-q 
\end{smallmatrix}\right)$.

\begin{lemme}\cite{Left}
	\label{L:Enlacement}
	Pour $k_1, k_2$ deux entrelacs de~$\Sph^3$ disjoints et disjoints de l'entrelacs de Hopf~\Hopf, leur nombre d'enlacement après avec effectué des chirurgies de pentes~$(p-1, q-1, r-1)$ sur les trois composantes~$H_1, H_2, H_3$ de~\Hopf~est donné par
	\[\Enl_{\U\Spqr}(k_1,k_2) = \Enl_{\Sph^3}(k_1, k_2) + Q_{p,q,r}
	\left( 
		\left(\begin{smallmatrix}
			\Enl_{\Sph^3}(k_1, H_1)\\
			\Enl_{\Sph^3}(k_1, H_2)\\
			\Enl_{\Sph^3}(k_1, H_3)
		\end{smallmatrix}\right),
		\left(\begin{smallmatrix}
			\Enl_{\Sph^3}(k_2, H_1)\\
			\Enl_{\Sph^3}(k_2, H_2)\\
			\Enl_{\Sph^3}(k_2, H_3)
		\end{smallmatrix}\right)
	\right)
	.\]
\end{lemme}

\subsection{Preuve des propositions~\ref{P:hEnl} et~\ref{P:3huitEnl}}

Le premier point des propositions~\ref{P:hEnl} et~\ref{P:3huitEnl} a en fait déjà été démontré~\cite{Left} puisqu'il y est démontré que deux orbites \emph{quelconques} de~$\fdts$ ou de~$\fttq$ ont un enlacement négatif. 

On va néanmoins expliquer la preuve de la proposition~\ref{P:hEnl}. 
Pour deux orbites~$k_1, k_2$ du patron~$\mathcal{T}_{2,3,7}$ codées par des mots~$w_1, w_2$, on peut spécifier le lemme~\ref{L:Enlacement}~: 
l'enlacement $\Enl_{\U\Spqr}(k_1,k_2)$ est la somme de l'enlacement $\Enl_{\Sph^3}(k_1,k_2)$ qu'on voit sur le patron et qui est toujours un nombre négatif, et du terme bilinéaire qu'on peut ici spécifier~: en notant~$a_i$ le nombre de lettres $a$ dans~$w_i$ et $b_i$ le nombre de lettres $b$ dans~$w_i$, ce terme vaut $11a_1a_2-7(a_1b_2+b_1a_2)+5b_1b_2$~: il est toujours positif. 

Il se trouve que $\Enl_{\Sph^3}(k_1,k_2)$ est presque bilinéaire en les codes~$w_1, w_2$. 
Plus précisément, si $k_1$ est l'orbite~$\vec h$, alors~$w_1$ est le code~$(xy)^\Z=(ababb)^\Z$. 
Si $k_2$ est une orbite dont le code contient au moins deux syllabes en~$x^*y^*$, c'est-\`a-dire est différent de~$(xy)^\Z, (xy^2)^\Z, (x^2y)^\Z, (x^2y^2)^\Z$, alors on peut trouver deux orbites $k_3, k_4$ de~$\mathcal{T}_{2,3,7}$ de codes~$w_3^\Z$, $w_4^\Z$ telles que $w_2=w_3w_4$, $k_2$ est homotope dans le complément de~$k_1=\vec h$ à la somme connexe de~$k_3$ et $k_4$, et donc en particulier $\Enl_{\Sph^3}(\vec h,k_2)=\Enl_{\Sph^3}(\vec h,k_3)+\Enl_{\Sph^3}(\vec h,k_4)$. 

Ainsi il suffit de vérifier que $\vec h$ s'enlace négativement avec les orbites de codes~$(xy)^\Z, (xy^2)^\Z, (x^2y)^\Z, (x^2y^2)^\Z$ pour vérifier qu'elle s'enlace négativement avec toutes les orbites périodiques de~$\mathcal{T}_{2,3,7}$. 
L'utilisation du lemme~\ref{L:Enlacement} donne alors 
\[\Enl_{\U\Sdts}(\vec h, (xy)^\Z) = \Enl_{\U\Sdts}(\vec h, (x^2y)^\Z) = \Enl_{\U\Sdts}(\vec h, (xy^2)^\Z) = -1\quad\mathrm{et}\quad\Enl_{\U\Sdts}(\vec h, (x^2y^2)^\Z)=-2,\]
ce qui démontre le premier point de la proposition~\ref{P:hEnl}.

Pour le second point, il s'agit de calculer un auto-enlacement. 
Il se trouve que le framing donné par la direction stable du flot d'Anosov est isotope au framing donné par la direction du patron~$\mathcal{T}_{2,3,7}$. 
On a alors comme précédemment
\[\Enl_{\U\Sdts}(\vec h, \vec h^{+s})=-1.\]

Le troisième point a en fait été démontré en même temps que le premier: 
puisque l'enlacement avec~$\vec h$ peut \^etre rendu linéaire, on a vu que la seule orbite de~$\fdts$ ayant enlacement~$-1$ avec~$\vec h$ est $\vec h$ elle-même. 

\medskip
La proposition~\ref{P:3huitEnl} se démontre de la même fa\c con, mais il faut tenir compte de contraintes légèrement différentes sur les codes. 

%\subsection{Calculs des coefficients de chirurgie}\label{CalculChirurgie}
%
%On va montrer que les propositions~\ref{P:hEnl} et~\ref{P:3huitEnl} impliquent un peu plus que le théorème~\ref{Th}, à savoir
%
%\begin{prop}\label{P:CalculChirurgie}
%\begin{enumerate}
%\item 
%La variété~$\U\Sdts$ est homéomorphe \`a~$\Sph^3(~\Huit, 1)$. 
%\item
%La variété~$\U\Sttq$ est homéomorphe \`a~$\Sph^3(~\Huit, 3)$. 
%\end{enumerate}
%\end{prop}
%
%\begin{proof}
%Par rapport au théorème~\ref{Th}, il manque le calcul des coefficients de chirurgie. 
%
%Commen\c cons par comprendre la chirurgie entre la variété-suspension de~$(\T^2, \cat)$, notée~$\T^3_\cat$, et~$\Sph^3$ : 
%comme expliqué dans l'introduction, le n\oe ud~$\Huit$ borde un tore, qu'on note~$S_\Huit$, dans~$\Sph^3$. 
%Comme $\Sph^3\setminus\Huit$ fibre sur le cercle avec fibre~$S_\Huit$, on sait que $\T^3_\cat$ est la variété~$\Sph^3(\Huit, 0)$. 
%De plus, comme le bord~$\Huit$ est un point fixe pour l'application de premier retour, l'image de~$\Huit$ sous la chirurgie est l'orbite du point fixe de~$\cat$, c'est-à-dire l'orbite du point $(0,0)\in\T^2$. 
%
%Réciproquement, partant de~$\T^3_\cat$
%
%
%\end{proof}

%%%%%%%%%%%%%%%%%%%%%%%

\appendix

\section{Sections de Birkhoff de genre 1 pour les flots géodésiques}

Les propositions~\ref{P:h} et~\ref{P:3huit} sont des exemples de sections Birkhoff pour des flots géodésiques sur des orbisurfaces hyperboliques. 
Plusieurs énoncés du même type ont été démontrés dans les 40 dernières années.  
Dans cet appendice nous tentons un recensement.

%Un r\'esultat folklorique, mais r\'edig\'e seulement r\'ecemment par Foulon et Hasselblatt, stipule que dans le contexte pr\'ec\'edent, si la collection~$\gamma$ est \emph{remplissante} (c'est-\`a-dire que $\Sigma\setminus\gamma$ n'est constitu\'ee que de r\'egions simplement connexes), alors la vari\'et\'e~$\U\Sigma\setminus\vec\gamma$ est hyperbolique. 
%La preuve de Foulon-Hasselblatt, reprise du blog de Calegari, est post-Perelman, puisqu'elle proc\`ede par \'elimination: elle montre que la vari\'et\'e est atoro\"idale, et qu'elle ne peut \^etre un fibr\'e de Seifert. 
%Les \'enonc\'es que nous pr\'esentons donnent, dans des cas particuliers, une preuve pr\'e-Perelman, puisqu'on montre que la vari\'et\'e~$\U\Sigma\setminus\vec\gamma$ fibre sur le cercle, avec monodromie pseudo-Anosov; le th\'eor\`eme d'hyperbolisation donne alors l'hyperbolicit\'e. 
%
Le premier énoncé du genre consiste en une construction de Birkhoff, popularisée par Fried dans le cas d'une surface hyperbolique. 
Pour $\gamma=\gamma_1\cup\dots\cup\gamma_n$ une collection de g\'eod\'esiques sur une surface~$\Sigma$, on note $\vecgamma$ son \emph{relev\'e antith\'etique} (ou relev\'e sym\'etrique) dans~$\U\Sigma$, \`a savoir l'entrelacs ayant $2n$ composantes associ\'ees aux $n$ composantes de~$\gamma$ et aux deux orientations possibles pour chaque composante. 
Une collection de géodésiques est dite remplissante si son complémentaire ne contient aucune géodésique fermée. 

\begin{theo}\cite{Birkhoff, FriedAnosov}\label{BirkhoffFried}
Si $\Sigma$ est une surface hyperbolique orbifoldique et $\gamma$ une collection remplissante de g\'eod\'esiques p\'eriodiques de~$\Sigma$, alors $\vecgamma$ borde une section de Birkhoff pour~$\fgeod$ dans~$\U\Sigma$, avec multiplicité~$-1$ pour chaque composante de bord. 
\end{theo}

Notons que la variété~$\U\Sigma$ n'étant pas une sphère d'homologie, un entrelacs peut border plusieurs surfaces non homologues. 
Dans le contexte du théorème~\ref{BirkhoffFried} les différentes sections de Birkhoff bordées par~$\vecgamma$ ont été classifiées avec Cossarini~\cite{DC}.

Il est naturel de chercher parmi les constructions précédentes lesquelles donnent des sections de petit genre. 
Il se trouve que le genre~$0$ est exclu pour les flots géodésiques (car les feuilletages stables et instables sont orientables, et une surface de genre 0 ne porte par de feuilletages orientables n'ayant pour singularités que des demi-selles au bord). 
On peut donc chercher les sections de genre 1. 

\begin{theo}\cite[théorème 1.4]{DO}\label{T:DO}
Sur une surface hyperbolique~$\Sigma$ de genre~$g$, il n'y a que 3 collections~$\gamma$ de géodésiques périodiques, modulo action du groupe modulaire~$\mathrm{MCG(\Sigma)}$, telles que $-\vecgamma$ borde une section de Birkhoff de genre 1 pour~$\fgeod$. 
\end{theo}

Le th\'eor\`eme~\ref{BirkhoffFried} n'est pas le plus g\'en\'eral possible puisqu'on a fait l'hypoth\`ese qu'on rel\`eve~$\gamma$ de fa\c con antith\'etique d'une part, et que les multiplicités sont~$-1$ sur chaque composante de bord. 
Si on consid\`ere une collection de g\'eod\'esiques orient\'ees~$\gamma$ et son relev\'e~$\vec\gamma$, on constate empiriquement que la vari\'et\'e fibre souvent (calculs personnels), mais pas toujours (contre-exemple de Th\'eo Marty). 
D'o\`u la question ouverte suivante:

\begin{ques}\label{ques}
Pour quelles collections remplissantes~$\gamma$ de g\'eod\'esiques orient\'ees la vari\'et\'e~$\U\Sigma\setminus\vec\gamma$ fibre-t-elle sur~$\Sph^1$? Quand la fibre est-elle de genre~$1$?
\end{ques}

Si la réponse aux deux questions est oui, la monodromie est en fait de type Anosov \`a feuilletages orientables, elle peut donc \^etre repr\'esent\'ee par une classe de conjugaison de matrices de~$\SLZ$ comme cela a été fait pour démontrer le théorème~\ref{Th}. 
Voici les constructions existantes, les trois premières correspondent au théorème~\ref{T:DO}.
Pour l'énoncé, on note~$\Sigma_g$ une surface hyperbolique de genre~$g$, et~$\Sigma_{g;p_1, \dots, p_n}$ une orbisurface hyperbolique de genre~$g$ avec points coniques d'ordre~$p_1, \dots, p_n$. 
Notons que les énoncés sont indépendants de la métrique, pourvu qu'elles soit hyperbolique. 
En effet, un théorème de Gromov~\cite{Gromov} affirme que les flots géodésiques associés à deux métriques hyperboliques distinctes sont en fait topologiquement conjugués. 

\begin{theo}
\label{T:Exemples}
Dans les cas suivants, la collection~$\vec\gamma$ borde une section de Birkhoff de genre 1 pour le flot géodésique, le nombre de composantes de bord est indiqu\'e, et la classe de conjugaison de l'application de premier retour est d\'ecrite en termes des g\'en\'erateurs standards~$L=(\begin{smallmatrix}1&0\\1&1\end{smallmatrix})$ et $R=(\begin{smallmatrix}1&1\\0&1\end{smallmatrix})$ de~$\SLZ$ :

\[
\begin{array}{c|c|c|c|c}
\Sigma & |\vec\gamma| & \mathrm{premier~retour} & \mathrm{sources}\\
\hline
\hline
\Sigma_{g}\quad (g\ge2) & 4g+4& L^{g-1}R^2L^{g-1}R^2 & \mathrm{[Bir17, Fri82, Ghy87, Has90]~(a)} \\
\hline
\Sigma_{g}\quad (g\ge2) & 4g+2& L^{g-1}R^4L^{g-1}R^2 & \mathrm{[Bir17, Bru94]~(b)} \\
\hline
\Sigma_{g}\quad (g\ge2) & 4g& L^{g-1}R^4L^{g-1}R^4  &  \mathrm{[Bir17, DeO21, DeL19]~(c)}\\
\hline
\hline
\Sigma_{0	;2,3,r}\quad (r\ge7) & 1& L^{r-6}R & \mathrm{[Deh15]}\\
\hline
\Sigma_{0	;2,q,r}\quad (q\ge4 ,r\ge5) & 1& L^{q-4}RL^{r-4}R &\mathrm{[Deh15]} \\
\hline
\Sigma_{0;p,q,r}\quad (p,q\ge3,r\ge4) & 1& L^{p-3}RL^{q-3}RL^{r-3}R & \mathrm{[Deh15]} \\
\hline
\Sigma_{0;p,q,r}\quad (p,q\ge3,r\ge4) & 2& L^{p+q-6}RL^{r-3}R & \mathrm{[HaM13, DeL19]~(d)} \\
\hline
\Sigma_{0;p,q,r,s}\quad (p,q,r\ge2,s\ge3) & 2& L^{p-2}RL^{q-2}RL^{r-2}RL^{s-2}R & \mathrm{[Deh15]} \\
\hline
\hline
\Sigma_{0;p_1,\dots,p_n}\quad (p_i\ge 4)\quad & n&? & \mathrm{[HaM13]} \\
\hline
\Sigma_{0;p_1,\dots,p_{2n}}\quad & 2n-2&? & \mathrm{[HaM13]} \\
\hline
\hline
\Sigma_{g;p_1,\dots,p_n}\quad (g\ge 1) & 4g+n+3&? & \mathrm{[DeS19]} \\
\end{array}
\]
\end{theo}

Pour les sources, (a) reprend la construction de Birkhoff qui a été popularisée par Fried. Le calcul du premier retour a été fait en deux temps: Ghys pour la trace, puis Hashiguchi pour la matrice complète. 
(b) reprend aussi la construction de Birkhoff, mais à partir de courbes découvertes par Brunella, qui a aussi calculé le premier retour. 
(c) reprend encore cette construction, mais avec une autre collection de courbes (remarquée indépendamment par Bonatti), le premier retour peut être calculé selon la méthode présentée dans l'article avec Liechti. 
Enfin (d) a été présenté par Hashiguchi et Minakawa, et le premier retour calculé avec Liechti. 
Pour les trois derniers exemples, le premier retour n'a pas été calculé. 

Notons que ce th\'eor\`eme n'est absolument pas exhaustif. 
En effet, on a

\begin{theo}(Minakawa, voir \cite{DShannon})\label{Minakawa}
Si $(\phi^t)_{t\in\R}$ est un flot d'Anosov qui admet une section de Birkhoff de genre~$1$, alors pour chaque matrice hyperbolique~$A\in\SLZ$ il admet une section dont le premier retour est donné par~$A$. 
\end{theo}

En particulier si un flot admet une section de Birkhoff de genre 1, il en admet une infinité. 
La dernière ligne du tableau précédent affirme que tout flot géodésique sur une orbisurface hyperbolique admet une section de Birkhoff de genre~$1$, et donc le théorème de Minakawa implique qu'il en admet une infinité. 
Cela ne clôt pas l'intérêt de la question~\ref{ques} puisque le nombre de composantes de bord dans les sections de Birkhoff du théorème~\ref{Minakawa} a tendance à cro\^itre rapidement.  
En particulier, on ne sait pas si on a recensé en~\ref{T:Exemples} toutes les collections ne comptant qu'une orbite périodique et bordant une section de Birkhoff de genre~$1$ -- probablement pas. 

%Les trois \'enonc\'es ci-dessus passent par la m\^eme id\'ee initiale (d\^ue \`a Birkhoff et red\'ecouverte par Fried): une \emph{section de Birkhoff} pour un flot dans une vari\'et\'e de dimension 3 est une surface compacte plong\'ee, dont l'int\'erieur est transverse au flot, dont le bord est constitu\'e d'un nombre fini d'orbites p\'eriodiques, et telle que toute orbite du flot coupe la surface en temps uniform\'ement born\'e. 
%Si dans une vari\'et\'e~$M$, un flot~$\phi$ admet une section de Birkhoff~$S$, alors il y a une application~$f_S$ de premier retour le long de~$\phi$ induite sur~$S$. 
%Dans ce cas la vari\'et\'e $M\setminus\partial S$ fibre sur le cercle, avec fibre~$S$ et monodromie~$f_S$. 
%Ainsi, dans les trois th\'eor\`emes ci-dessus, les fibres sont \`a chaque fois des sections de Birkhoff pour le flot g\'eod\'esique, et la monodromie est donn\'ee par l'application de premier retour sur ces sections. 
%
%Par exemple, pour la vari\'et\'e~$\U\Sigma_{0; 3, 3, 4}$, le th\'eor\`eme~\ref{T:Exemples} est impliqu\'e par l'existence de deux sections de Birkhoff de genre 1 pour le flot g\'eod\'esique, la premi\`ere ayant une composante de bord et un premier retour de la forme~$LR^3=\begin{pmatrix}1&3\\1&4\end{pmatrix}$, et la seconde ayant deux bords et un premier retour de la forme~$LR^2=\begin{pmatrix}1&2\\1&3\end{pmatrix}$. 
%
%Ainsi, pour d\'emontrer chacun des \'enonc\'es ci-dessus, il s'agit de trouver une section de Birkhoff, et d'\'etudier l'application de premier retour induite.
%
%

%%%%%%%%%%%%%%%%%%%%%%%%%%%%%

\bibliographystyle{siam}

\end{document}